\documentclass[english]{amsart}
\usepackage[T1]{fontenc}
\usepackage[latin9]{inputenc}
\usepackage{babel}
\usepackage{verbatim}
\usepackage{units}
\usepackage{mathrsfs}
\usepackage{amsbsy}
\usepackage{amstext}
\usepackage{amsthm}
\usepackage{amssymb}
\usepackage{graphicx}
\usepackage[a4paper]{geometry}
\usepackage[pdfusetitle,
 bookmarks=true,bookmarksnumbered=false,bookmarksopen=false,
 breaklinks=false,pdfborder={0 0 1},backref=false,colorlinks=false]
 {hyperref}

\makeatletter
\numberwithin{equation}{section}
\numberwithin{figure}{section}
\theoremstyle{plain}
\newtheorem{thm}{\protect\theoremname}[section]
\theoremstyle{definition}
\newtheorem{defn}[thm]{\protect\definitionname}
\theoremstyle{plain}
\newtheorem{prop}[thm]{\protect\propositionname}
\theoremstyle{plain}
\newtheorem{cor}[thm]{\protect\corollaryname}
\theoremstyle{plain}
\newtheorem{lem}[thm]{\protect\lemmaname}

\usepackage{upgreek}
\usepackage{amssymb}
\usepackage{graphicx}
\let\originalleft\left
\let\originalright\right
\renewcommand{\left}{\mathopen{}\mathclose\bgroup\originalleft}
\renewcommand{\right}{\aftergroup\egroup\originalright}

\makeatother

\providecommand{\corollaryname}{Corollary}
\providecommand{\definitionname}{Definition}
\providecommand{\lemmaname}{Lemma}
\providecommand{\propositionname}{Proposition}
\providecommand{\theoremname}{Theorem}

\begin{document}
\title{Differential Geometry on Pointwise Affine Spaces}
\author{Dan Jonsson}
\begin{abstract}
We introduce an alternative formalization of curved spaces in which
the concept of a pointwise affine space, as defined here, replaces
that of a manifold. New or modified definitions of familiar notions
from differential geometry such as connections, torsion, Riemann curvature,
vector fields and differentiation of vector fields are presented,
and examples of applications of the resulting framework are given.
In general, the notions and results considered receive simpler and
more transparent forms than in conventional approaches.
\end{abstract}

\maketitle

\section{Introduction}

It may be said that there are two mathematical languages for differential
geometry: one coordinate-oriented and one coordinate-free. 

The coordinate-oriented language is the oldest one, and bears evidence
of its origins. During the early development of differential geometry,
modern notions of abstract spaces had not yet crystallized. The first
axiomatic definition of affine spaces was given by Weyl in 1918 \cite{key-1},
and in spite of early contributions by Grassmann \cite{key-2} and
Peano \cite{key-3} abstract vector space theory ''was actually not
developed before 1920'' \cite[p. 228]{key-4}. By then, differential
calculus on curved spaces, following Ricci and Levi-Civita \cite{key-5},
had already been used in the theory of General Relativity. Differential
calculus on $\mathbb{R}^{n}$ regarded as a concrete flat space had
morphed into a more modern differential calculus when $\mathbb{R}^{n}$
came to be regarded as a possibly curved space. To clarify this point,
we give a brief and simplified account of generalized differential
calculus on $\mathbb{R}^{n}$.

$\mathbb{R}^{n}$ can be regarded as a space where points, vectors
and their coordinates live. We define a point $x\in\mathbb{R}^{n}$
as a tuple $\left(x_{1},\ldots,x_{n}\right)$, a vector $v\in\mathbb{R}^{n}$
as a tuple $\left(v_{1},\ldots v_{n}\right),$ and denote their their
coordinates by $\left(x^{1},\ldots,x^{n}\right)$ and $\left(v^{1},\ldots v^{n}\right)$,
respectively. A coordinate system is a bijection
\begin{gather*}
\mathsf{C}:\mathbb{R}^{n}\times\mathbb{R}^{n}\rightarrow\mathbb{R}^{n}\times\mathbb{R}^{n},\\
\left(\left(x_{1},\ldots,x_{n}\right),\left(v_{1},\ldots,v_{n}\right)\right)\mapsto\left(\left(x^{1},\ldots,x^{n}\right),\left(v^{1},\ldots,v^{n}\right)\right)
\end{gather*}
such that $x^{i}=f^{i}\left(x\right)$ and $v^{i}=f^{i}\left(v\right)$
for $i=1,\ldots,n$. A Cartesian coordinate system is a coordinate
system such that $x^{i}=f^{i}\left(x_{1},\ldots,x_{n}\right)=x_{i}$
and $v^{i}=f^{i}\left(v_{1},\ldots,v_{n}\right)=v_{i}$ for $i=1,\ldots,n$.
In particular, we let
\[
v\thicksim_{\mathsf{C}}v^{k}
\]
 describe a situation where the vector $v$ has coordinates $\left(v^{1},\ldots,v^{n}\right)$
relative to $\mathsf{C}$.

A vector field $\boldsymbol{v}$ on $\mathbb{R}^{n}$ is a function
\[
\boldsymbol{v}:\mathbb{R}^{n}\rightarrow\mathbb{R}^{n},\qquad x\mapsto v.
\]
Using the Einstein summation convention, we can define the directional
derivative $\nabla_{u}\boldsymbol{v}\left(x\right)$ of the differentiable
vector field $\boldsymbol{v}$ with respect to the vector $u$ at
the point $x$ by requiring that
\begin{equation}
\nabla{}_{u}\boldsymbol{v}\left(x\right)\sim_{C}u^{i}\frac{\partial\boldsymbol{v}^{k}}{\partial x^{i}}\left(x\right),\label{eq:coorder}
\end{equation}
where $C$ is a Cartesian coordinate system. However, if we substitute
coordinates relative to a curvilinear (that is, not necessarily Cartesian)
coordinate system $\overline{C}$ in (\ref{eq:coorder}) then the
definition of $\nabla{}_{u}\boldsymbol{v}\left(x\right)$ may need
to be given another form since, in general,
\[
\nabla_{u}\boldsymbol{v}\left(x\right)\nsim_{\overline{C}}\,\overline{u}^{i}\frac{\partial\overline{\boldsymbol{v}}^{k}}{\partial\overline{x}^{i}}\left(x\right).
\]
We thus need to complement (\ref{eq:coorder}) by a correction term,
which we express by means of a system of scalars $\Gamma_{ij}^{k}$
for each point $x$, so that we obtain a coordinate-system independent
definition
\begin{equation}
\nabla_{u}\boldsymbol{v}\left(x\right)\sim_{\overline{C}}\,\overline{u}^{i}\frac{\partial\overline{\boldsymbol{v}}^{k}}{\partial\overline{x}^{i}}\left(x\right)+\overline{u}^{i}\overline{\boldsymbol{v}}^{j}\Gamma_{ij}^{k}\left(x\right).\label{eq:curvi}
\end{equation}
By the construction of $\Gamma_{ij}^{k}$, we have $\Gamma_{ij}^{k}\left(x\right)=0$
for all $i,j,k,x$ if $\overline{C}$ is Cartesian, so $\Gamma_{ij}^{k}$
measures the discrepancy between $\overline{C}$ and a Cartesian coordinate
system.

We have thus generalized (\ref{eq:coorder}), a formula for calculating
a directional derivative in $\mathbb{R}^{n}$, valid only for a Cartesian
coordinate system, to a formula (\ref{eq:curvi}), valid for a curvilinear
coordinate system. There is, however, another way to look at (\ref{eq:curvi}).
In Einstein's and Grossmann's \emph{Entwurf} article \cite[p. 23]{key-6},
the latter makes a profound observation:
\begin{quote}
{\small Ricci und Levi-Cività haben, ausgehend von den Christoffel\textquoteright schen
Resultaten, ihre Methoden der absoluten, d.h. vom Koordinatensystem
unabhängigen Differentialrechnung entwickelt, die gestatten, den Differentialgleichungen
der mathematischen Physik eine invariante Form zu geben. Da aber die
}{\small\emph{Vektoranalysis des}}{\small{} }{\small\emph{auf beliebige
krummlinige Koordinaten bezogenen euklidischen Raumes}}{\small{} formal
identisch ist mit der}{\small\emph{ Vektoranalysis einer beliebigen
}}{\small{[}---{]} }{\small\emph{Mannigfaltigkeit}}{\small , so bietet
es keine Schwierigkeiten, die vektoranalytischen Begriffsbildungen
{[}---{]} auszudehnen auf die vorstehende allgemeine Theorie von
Einstein. {[}emphasis supplied{]} }{\small\footnote{{\small Starting from the results of Christoffel, Ricci and Levi-Cività
have developed their methods of absolute, i.e., coordinate-system
independent, differential calculus that allows the differential equations
of mathematical physics to be given an invariant form. However, as
}\emph{vector analysis in a Euclidean space with arbitrary curvilinear
coordinate}{\small s is formally identical to }\emph{vector analysis
in an arbitrary manifold}{\small , it does not present any difficulties
to extend the conceptual apparatus of vector analysis {[}so as to
apply to{]} Einstein's general theory presented above.}}}{\small\par}
\end{quote}
In other words, the change of coordinates of  $u$ and $v$ can depend
on a change of the coordinate system (''passive transformation''),
or it can depend on a change of the vectors $u$ and $v$ themselves
(''active transformation''). These are interchangeable points of
view. In the latter case, the vectors are transformed as a consequence
of a deformation of $\mathbb{R}^{n}$ to a curved space. Thus, in
the first case, $\Gamma_{ij}^{k}$ describe the ''curvature'' of
the coordinate system; in the second case $\Gamma_{ij}^{k}$ describe
the curvature of the space itself.

Thus, early differential geometry was characterized by
\begin{itemize}
\item a coordinate-oriented mathematical language,
\item a setting for differential calculus in the form of $\mathbb{R}^{n}$
regarded as a possibly curved space,
\item a tool $\Gamma_{ij}^{k}$, nowadays known as a connection, for describing
the curvature of $\mathbb{R}^{n}$.
\end{itemize}
The coordinate-oriented approach, as amplified by the tensor formalism
\cite{key-5}, is quite powerful, but has two significant weaknesses.
The first one is that $\mathbb{R}^{n}$ has a globally trivial topology
(shape): it cannot naturally represent a sphere, a torus etc. The
second weakness derives precisely from the fact that a coordinate-oriented
mathematical language is used. For example, in modern algebra a vector
is conceived as simply an element of a vector space, not, in a more
convoluted way, as an object the coordinates of which change in a
particular way when the coordinate system changes in a particular
way. A related reproach is that the ''débauches d\textquoteright indices''
(Cartan) used in the coordinate-oriented language tends to obscure
the geometric meaning of objects and operations.

Thus, the emphasis in the further development of differential geometry
has fallen on elaborating an appropriate notion of a manifold that
generalizes $\mathbb{R}^{n}$ so as to accommodate spaces with more
general shapes, and on proposing coordinate-free formulations whenever
feasible. Several sophisticated frameworks have been developed, characterized
by
\begin{itemize}
\item an abstract, algebraically oriented mathematical language,
\item a setting for differential calculus in the form of a fiber bundle
(such as a tangent bundle, a vector bundle, a frame bundle or a principal
bundle) the base space of which is a differentiable manifold, a generalization
of $\mathbb{R}^{n}$ (sometimes $\mathbb{C}^{n})$ that is locally
but not necessarily globally similar to $\mathbb{R}^{n}$ ($\mathbb{C}^{n})$,
\item an abstract notion of a connection on the fiber bundle involved.
\end{itemize}
The price in the form of increased complexity paid to generalize $\mathbb{R}^{n}$
and reduce the use of coordinates is, alas, high. While this increased
complexity serves a purpose for advanced applications, there are many
simpler applications where it is redundant. Worse, a significant disadvantage
of the coordinate-oriented approach persists: the geometric meaning
of abstract algebraic expressions is often as obscure as that of coordinate
expressions. Is it really not possible to find a simpler and simultaneously
more transparent coordinate-free formulation of differential geometry?%

Following up on previous work \cite{key-7}, where the key notion
of a pointwise affine space was introduced, this is the question that
will be addressed in this article. In the approach proposed here,
the main elements are
\begin{itemize}
\item a coordinate-free, geometric-algebraic mathematical language,
\item a setting for differential calculus in the form of a pointwise affine
space,
\item a geometrically meaningful connection, defined via a pseudo-derivative
$\Delta_{u}v$ on the pointwise affine space.
\end{itemize}
The price paid for simplicity and transparency is that the present
formulation again applies only to spaces with a simple global topology,
although a generalization to more general cases may be possible (see
Section 10).

In Section 2, we review affine (flat) spaces, and then generalize
to pointwise affine (flat and curved) spaces in Section 3. Using notions
from Section 3, we introduce tools for analyzing discrete and infinitesimal
curvature in Section 4. In Section 5, we then show how torsion and
Riemann curvature can be defined and interpreted in terms of the new
notions, in particular the coordinate-free pseudo-derivative $\Delta_{u}v$.
In Sections 6 -- 8, we define and investigate differentiation of
functions that connect pointwise affine spaces; this includes differentiation
of vector fields. Section 9 contains applications of the theory developed
to some topics of interest in differential geometry. In Section 10,
we make some final remarks, providing a perspective on the theory
of pointwise affine spaces.

\section{\label{sec:Fundamental-notionss2}Affine spaces}

Recall that there are two equivalent ways to define an action of a
group $G$ on a set $X$, in particular, an action of the additive
group of a vector space $V$ on an affine space. We can use a function
$\upphi:G\times X\rightarrow X$ such that $\upphi\left(e,x\right)=x$
and $\upphi\left(g,\upphi\left(h,x\right)\right)=\upphi\left(gh,x\right)$,
where $e$ is the identity in $G$. Then we write $\upphi\left(g,x\right)$
as $g\cdot x$, or $\upphi\left(v,x\right)$ as $x+v$ in the case
of a vector space. Alternatively, we may use a group homomorphism
$\phi:G\rightarrow\mathcal{S}{}_{X}$, where $\mathcal{S}{}_{X}$
is the group of bijections on $X$. Following Berger \cite{key-8},
we use the second approach here.
\begin{defn}
\label{d1}Let $A$ be a non-empty set, $\mathcal{S}_{A}$ the group
of all bijections on $A$ and $V$ a vector space over a field $K$.
Also let $\mathfrak{a}:V\rightarrow\mathcal{S}_{A}$ be a group homomorphism
with $V$ regarded as an abelian group, so that $u+v=v+u$, $\mathfrak{a}\left(0\right)=\mathrm{id}_{\mathcal{S}_{A}}$
and $\mathfrak{a}\left(u+v\right)=\mathfrak{a}\left(u\right)\circ\mathfrak{a}\left(v\right)$
for all $u,v\in V$. $\mathfrak{a}$ is said to be an \emph{affine
action} of $V$ on $A$ when
\begin{enumerate}
\item $\mathfrak{a}$ is injective: if $u,v\in V$ and $\mathfrak{a}\left(u\right)=\mathfrak{a}\left(v\right)$
then $u=v$\emph{;}
\item $\mathfrak{a}$ is transitive: if $p,q\in A$ then there is some $v\in V$
such that $\mathfrak{a}\left(v\right)\left(p\right)=q$. 
\end{enumerate}
An \emph{affine system} is a tuple $\mathcal{A}=\left(V,A,\mathfrak{a}\right)$,
where $\mathfrak{a}:V\rightarrow\mathcal{S}_{A}$ is an affine action
of $V$ on $A$. We denote the set of all affine actions of $V$ on
$A$ by $V\left(A\right)$. The elements of $A$ are called \emph{points},
and $A$ is called a \emph{point space} or, in the context of an affine
system, an \emph{affine space}.
\end{defn}

As the additive group of $V$ is abelian, it follows from (1) and
(2) that $\mathfrak{a}$ is \emph{simply transitive}, that is, for
any two two $p,q\in A$ there is exactly one $v\in V$ such that $\mathfrak{a}\left(v\right)\left(p\right)=q$
\cite[Proposition 1.4.4.1]{key-8}. Conversely, if $\mathfrak{a}$
is simply transitive then $\mathfrak{a}$ is injective and transitive.

Below, $\mathfrak{a}$ will denote the surjective homomorphism $V\rightarrow\mathfrak{a}\left(V\right)$
rather than the homomorphism $V\rightarrow\mathcal{S}_{A}$. As $\mathfrak{a}$
is injective as well, it is a group isomorphism, and we can even extend
$\mathfrak{a}$ to a vector space isomorphism. 
\begin{defn}
Let $V$ be a vector space over $K$ and consider the function 
\[
\mathfrak{s}:K\times\mathfrak{a}\left(V\right)\rightarrow\mathfrak{a}\left(V\right),\qquad\left(\lambda,\mathfrak{a}\left(v\right)\right)\mapsto\lambda\left(\mathfrak{a}\left(v\right)\right):=\mathfrak{a}\left(\lambda v\right).
\]
$\lambda\left(\mathfrak{a}\left(v\right)\right)$ is called the \emph{scalar
product }of\emph{ $\mathfrak{a}\left(v\right)$ by $\lambda$}. We
let $\vec{V}$ denote $\mathfrak{a}\left(V\right)$ equipped with
$\mathfrak{s}$, while $\vec{v}$ denotes $\mathfrak{a}\left(v\right)$
as an element of $\vec{V}$. We also call the elements of $\vec{V}$
\emph{translations} and $\vec{V}$ a \emph{translation space, }while
$V$ may be called \emph{a ground space.}
\end{defn}

\begin{prop}
$\vec{V}$ is a vector space isomorphic to $V$.
\end{prop}

\begin{proof}
The function $\mathfrak{s}$ defines a scalar multiplication on $\vec{V}$
since
\begin{align*}
1\left(\mathfrak{a}\left(v\right)\right) & =\mathfrak{a}\left(1v\right)=\mathfrak{a}\left(v\right),\\
\kappa\lambda\left(\mathfrak{a}\left(v\right)\right) & =\mathfrak{a}\left(\left(\kappa\lambda\right)v\right)=\mathfrak{a}\left(\kappa\left(\lambda v\right)\right)=\kappa\left(\mathfrak{a}\left(\lambda v\right)\right)=\kappa\left(\lambda\left(\mathfrak{a}\left(v\right)\right)\right),\\
\lambda\left(\mathfrak{a}\left(u\right)\circ\mathfrak{a}\left(v\right)\right) & =\lambda\left(\mathfrak{a}\left(u+v\right)\right)=\mathfrak{a}\left(\lambda\left(u+v\right)\right)=\mathfrak{a}\left(\lambda\left(u\right)+\lambda\left(v\right)\right)\\
 & =\mathfrak{a}\left(\lambda\left(u\right)\right)\circ\mathfrak{a}\left(\lambda\left(v\right)\right)=\lambda\left(\mathfrak{a}\left(u\right)\right)\circ\lambda\left(\mathfrak{a}\left(v\right)\right),\\
\left(\kappa+\lambda\right)\left(\mathfrak{a}\left(v\right)\right) & =\mathfrak{a}\left(\left(\kappa+\lambda\right)v\right)=\mathfrak{a}\left(\kappa v+\lambda v\right)=\mathfrak{a}\left(\kappa v\right)\circ\mathfrak{a}\left(\lambda v\right)\\
 & =\kappa\left(\mathfrak{a}\left(v\right)\right)\circ\lambda\left(\mathfrak{a}\left(v\right)\right).
\end{align*}

Furthermore, $\mathfrak{a}:V\rightarrow\vec{V}$ is a vector space
isomorphism since it is a group isomorphism and we have $\mathfrak{a}\left(\lambda v\right)=\lambda\left(\mathfrak{a}\left(v\right)\right)$
by definition.
\end{proof}
To make formulas look more familiar, we will use the notational identities
\begin{gather}
\vec{v}\left(p\right)\equiv p+\vec{v},\qquad\vec{u}\circ\vec{v}\equiv\vec{u}+\vec{v}.\label{eq:nid}
\end{gather}
Note that as $\mathfrak{a}:V\rightarrow\vec{V}$ is an isomorphism,
we have 
\begin{gather*}
\vec{0}\left(p\right)=\mathrm{id}_{\vec{V}}\left(p\right)=p,\\
\vec{v}\left(\vec{u}\left(p\right)\right)=\left(\vec{v}\circ\vec{u}\right)\left(p\right)=\mathfrak{a}\left(v+u\right)\left(p\right)=\mathfrak{a}\left(u+v\right)\left(p\right)=\left(\vec{u}\circ\vec{v}\right)\left(p\right).
\end{gather*}
In the notation from (\ref{eq:nid}), we have
\begin{gather}
p+\vec{0}=p,\label{eq:21}\\
\left(p+\vec{u}\right)+\vec{v}=p+\left(\vec{v}+\vec{u}\right)=p+\left(\vec{u}+\vec{v}\right).\label{eq:22-1}
\end{gather}

Let $v$ and hence $\mathfrak{a}\left(v\right)$ be determined by
the condition $\text{ \ensuremath{\mathfrak{a}}\ensuremath{\left(v\right)\left(p\right)}}=q$
for given $p$ and $q$, and denote $\mathfrak{a}\left(v\right)$
by $q\leftarrow p$, so that by definition
\begin{equation}
\left(q\leftarrow p\right)\left(p\right)=p+\left(q\leftarrow p\right)=q\label{eq:23}
\end{equation}
Since $q\leftarrow p$ is the unique translation on $A$ that sends
$p$ to $q$, we have
\begin{gather}
\left(p\leftarrow p\right)=\vec{0},\label{eq:23-1}\\
\left(r\leftarrow q\right)\circ\left(q\leftarrow p\right)=\left(r\leftarrow q\right)+\left(q\leftarrow p\right)=r\leftarrow p\quad(\mathrm{Weyl's}\;\mathrm{axiom)}.\label{eq:23-2}
\end{gather}
Algebraically, $q\leftarrow p$ behaves like $q-p$ as evident from
(\ref{eq:23}) -- (\ref{eq:23-2}).

\pagebreak{}

As per a popular quip, an affine space is a vector space that has
forgotten its origin. Actually, we can always let a vector space act
on itself as an affine space. Specifically, we can interpret a vector
space $V$ as a point space $V$ in an affine system $\mathcal{V}=\left(V,V,\mathfrak{a}\right)$,
where the second $V$ denotes the underlying set of $V$ and $\mathfrak{a}:V\rightarrow\mathfrak{a}\left(V\right)$
is given by $\mathfrak{a}\left(v\right)=a+v$ for some $a\in V$.
Then the conditions in Definition \ref{d1} are satisfied since for
any $a,b,v\in V$ we have $0+v=v$ and $a+\left(b+v\right)=\left(a+b\right)+v$,
and for any $u,v\in V$ there is exactly one $a\in V$ such that $u=a+v$.
Hence, we can regard a vector space as an affine space with vectors
as points, being acted on by itself, and we have $u\leftarrow v=u-v$.
Note that there is a canonical isomorphism $\mathfrak{ci}_{V}:V\rightarrow\vec{V}$
such that $\mathfrak{ci}_{V}\left(a\right)\left(v\right)=a+v$ for
all $v\in V$.

Conversely, for every $p\in A$ there is a bijection 
\[
\mathfrak{v}\left(p\right):A\rightarrow V,\qquad q\mapsto\mathfrak{v}\left(p\right)\left(q\right):=\mathfrak{a}^{-1}\left(p\leftarrow q\right),
\]
such that $\mathfrak{v}\left(p\right)\left(p\right)=0$ and we can
make $A$ into a vector space $\vec{A}$ by requiring that $\mathfrak{v}\left(p\right)\left(q+r\right)=\mathfrak{v}\left(p\right)\left(q\right)+\mathfrak{v}\left(p\right)\left(r\right)$
and $\mathfrak{v}\left(p\right)\left(\lambda q\right)=\lambda\mathfrak{v}\left(p\right)\left(q\right)$;
$\mathfrak{v}\left(p\right)$ is then an isomorphism between $\vec{A}$
and $V$.

Before concluding the review of affine spaces, we introduce another
useful notion: an \emph{interval} is a pair of points $\left(p,q\right)$,
denoted $p\cdots q$. The interval $p\cdots q$ obviously determines
the unique translation that sends $p$ to $q$. One can think of an
interval $p\cdots q$ as a directed line segment from $p$ to $q$
or as a bound vector $v_{p}$ corresponding to the interval $p\cdots p+\vec{v}$,
in particular, a tangent vector at $p$ in an affine space and something
like a tangent vector at $p$ in classical differential geometry.

\section{Pointwise affine spaces}

Consider an affine system $\mathcal{A}=\left(V,A,\mathtt{\mathfrak{a}}\right)$
and recall the identity (\ref{eq:22-1}), according to which
\[
\left(\left(p+\vec{u}\right)+\vec{v}\right)\leftarrow\left(p+\left(\vec{u}+\vec{v}\right)\right)=\vec{0}
\]
in an affine space. To describe curved spaces we need to relax this
assumption. This can be done by instead assuming that $\mathcal{A}$
is affine in a particular way for each point of $A$ \cite[Sections 3.1--3.2]{key-7}.
\begin{defn}
An \emph{affine action field} of a vector space $V$ on a point space
$A$ is a function
\[
\boldsymbol{\mathfrak{a}}:A\rightarrow V\left(A\right),\qquad p\mapsto\boldsymbol{\mathfrak{a}}\left(p\right).
\]
 A \emph{pointwise affine system} is a tuple $\mathscr{A}\left(V,A,\boldsymbol{\mathfrak{a}}\right)$,
where $\boldsymbol{\mathfrak{a}}$ is an affine action field of $V$
on $A$. In the context of a pointwise affine system, we call the
point set $A$ a \emph{pointwise affine spac}e.
\end{defn}

For each $p\in A$, there is thus an isomorphism
\[
\boldsymbol{\mathfrak{a}}\left(p\right):V\rightarrow\vec{V},\qquad v\mapsto\vec{v}=\boldsymbol{\mathfrak{a}}\left(p\right)\left(v\right).
\]

For convenience, we may use the notation
\begin{equation}
p\left(v\right)\equiv\boldsymbol{\mathfrak{a}}\left(p\right)\left(v\right),\qquad p^{-1}\left(\vec{v}\right)\equiv\mathfrak{\boldsymbol{\mathfrak{a}}}\left(p\right)^{-1}\left(\vec{v}\right)\label{eq:ap-not}
\end{equation}
when it is not necessary to reference $\boldsymbol{\mathfrak{a}}$
explicitly. Employing this simplified notation, we have
\begin{gather*}
p\circ p^{-1}\left(\vec{v}\right)=\vec{v},\quad p^{-1}\circ p\left(v\right)=v,\\
p\left(\lambda u+\kappa w\right)=\lambda p\left(u\right)+\kappa p\left(v\right),\quad p^{-1}\left(\lambda\vec{u}+\kappa\vec{v}\right)=\lambda p^{-1}\left(\vec{u}\right)+\kappa p^{-1}\left(\vec{v}\right),
\end{gather*}
where $u,v\in V$ and $\vec{u},\vec{v}\in\vec{V}.$ We will also use
the notation
\begin{equation}
p+\overline{v}\equiv p+p\left(v\right).\label{eq:pbar}
\end{equation}
Thus, $\left(p+\overline{u}\right)+\overline{v}=\left(p+p\left(u\right)\right)+\left(p+p\left(u\right)\right)\left(v\right)$
etc.
\begin{defn}
Consider a pointwise affine system $\mathscr{A}\left(V,A,\boldsymbol{\mathfrak{a}}\right)$,
We say that $A$ is \emph{affine flat} when $p\left(v\right)=q\left(v\right)$
for any $p,q\in A$ and $v\in V$.
\end{defn}

An affine flat pointwise affine space $A$ is effectively an affine
space. Using the notation introduced in (\ref{eq:ap-not}) and (\ref{eq:pbar}),
we can form an expression that reflects to what extent and in what
way $A$ is curved at $p$, and vanishes everywhere if and only if
$A$ is affine flat.
\begin{prop}
\label{p8}A pointwise affine space $A$ is affine flat if and only
if 
\[
\left(\left(p+\overline{u}\right)+\overline{v}\right)\leftarrow\left(p+\left(p\left(u\right)+p\left(v\right)\right)\right)=\vec{0}
\]
 for all $p\in A$ and $u,v\in V$.
\end{prop}

\begin{proof}
If $p\left(v\right)=q\left(v\right)$ for every $p,q\in A$ and $v\in V$
then 
\begin{align*}
\left(p+\overline{u}\right)+\overline{v} & =p+p\left(u\right)+\left(p+p\left(u\right)\right)\left(v\right)=p+\left(p\left(u\right)+p\left(v\right)\right).
\end{align*}
Conversely, if $\left(p+\overline{u}\right)+\overline{v}=p+\left(p\left(u\right)+p\left(v\right)\right)$
then $p+p\left(u\right)+\left(p+p\left(u\right)\right)\left(v\right)=p+p\left(u\right)+p\left(v\right)$,
so $\left(p+p\left(u\right)\right)\left(v\right)=p\left(v\right)$
for all $p\in A$ and $u,v\in V$, and setting $u=p^{-1}\left(q\leftarrow p\right)$
we obtain
\[
q\left(v\right)=\left(p+p\left(p^{-1}\left(q\leftarrow p\right)\right)\right)\left(v\right)=\left(p+p\left(u\right)\right)\left(v\right)=p\left(v\right).\qedhere
\]
\end{proof}
\begin{cor}
\label{cor:c8}A pointwise affine space $A$ is affine flat if and
only if 
\begin{equation}
\left(p+\overline{u}\right)\left(v\right)-p\left(v\right)=\vec{0},\label{eq:conn-0}
\end{equation}
or equivalently 
\begin{equation}
p^{-1}\circ\left(p+\overline{u}\right)\left(v\right)-v=0,\label{eq:conn-1}
\end{equation}
 for all $p\in A$ and $u,v\in V$.
\end{cor}

\section{Discrete and infinitesimal curvature}

Throughout this section, we will consider a pointwise affine system
$\mathscr{A}=\left(V,A,\boldsymbol{\mathfrak{a}}\right)$. We first
consider curvature from the point of view of discrete differential
geometry, where we do not yet employ the notion of a limit.

\subsection{Deviation, dissociation and displacement}
\begin{defn}
A \emph{(vector)} \emph{deviation} is a function 
\begin{gather*}
\mathrm{G}:V\times V\times A\rightarrow V,\\
\left(u,v,p\right)\mapsto\mathrm{G}_{u}v\left(p\right):=p^{-1}\circ\left(p+\overline{u}\right)\left(v\right),
\end{gather*}
or a value of such a function. A \emph{(vector) dissociation} is a
function
\begin{gather*}
\mathrm{D}:V\times V\times A\rightarrow V,\\
\left(u,v,p\right)\mapsto\mathrm{D}_{u}v\left(p\right):=\mathrm{G}_{u}v\left(p\right)-v=p^{-1}\circ\left(p+\overline{u}\right)\left(v\right)-v,
\end{gather*}
or a value of such a function. Variants of these functions are 
\begin{gather*}
\vec{\mathrm{G}}:V\times V\times A\rightarrow\vec{V},\\
\left(u,v,p\right)\mapsto\vec{\mathrm{G}}_{u}v\left(p\right);=p\left(\mathrm{G}_{u}v\left(p\right)\right)=\left(p+\overline{u}\right)\left(v\right)
\end{gather*}
 and
\begin{gather*}
\vec{\mathrm{D}}:V\times V\times A\rightarrow\vec{V},\\
\left(u,v,p\right)\mapsto\vec{\mathrm{D}}_{u}v\left(p\right):=p\left(\mathrm{D}_{u}v\left(p\right)\right)=\left(p+\overline{u}\right)\left(v\right)-p\left(v\right).
\end{gather*}
\end{defn}

Remarkably, $\mathrm{D}_{u}v\left(p\right)$ is precisely the expression
(\ref{eq:conn-1}) that measures non-flatness of a pointwise affine
space. $\mathrm{D}_{u}v\left(p\right)$ and its infinitesimal counterpart
$\Delta_{u}v\left(p\right)$ feature prominently in the remainder
of this article. 
\begin{defn}
\label{def:42}Consider a pointwise affine system $\mathscr{A}\left(V,A,\boldsymbol{\mathfrak{a}}\right)$.
The function 
\[
\mathrm{C}\left(p\right):\left(u,v\right)\mapsto\mathrm{D}_{u}v\left(p\right)=p^{-1}\circ\left(p+\overline{u}\right)\left(v\right)-v
\]
is the \emph{discrete affine curvature} of $A$ at $p$, and a function
\[
\mathbf{C}:p\mapsto\mathrm{C}\left(p\right)
\]
is a \emph{discrete affine connection} on $A$.
\end{defn}

Intuitively, a deviation $\mathrm{G}$ in a curved space is, in general,
a rotation and rescaling of a vector bound to a fixed point, due to
a deformation of the space, while a dissociation $\mathrm{D}$ is
the discrepancy between the original vector and the rotated/rescaled
vector.

It follows directly from the definitions of $\mathrm{G}$ and $\mathrm{D}$
that $\mathrm{G}_{u}v\left(p\right)$, $\mathrm{D}_{u}v\left(p\right)$,
$\vec{\mathrm{G}}_{u}v\left(p\right)$ and $\vec{\mathrm{D}}{}_{u}v\left(p\right)$
are linear in $v$ for all $p\in A$ since $\mathfrak{a}\left(p\right)$
and $\mathfrak{a}\left(p\right)^{-1}$ are linear functions. Hence,
$\mathrm{D}_{u}0\left(p\right)=0$ and $\vec{\mathrm{D}}_{u}0\left(p\right)=\vec{0}$.
However, $\mathrm{G}_{u}v\left(p\right)$, $\mathrm{D}_{u}v\left(p\right)$,
$\vec{\mathrm{G}}_{u}v\left(p\right)$ and $\vec{\mathrm{D}}{}_{u}v\left(p\right)$
are not necessarily linear in $u$, although clearly $\mathrm{D}_{0}v\left(p\right)=0$
and $\vec{\mathrm{D}}_{0}v\left(p\right)=\vec{0}$.

Deviation and dissociation functions can be composed. Specifically.
\begin{gather}
\mathrm{G}_{u}\left(\mathrm{G}_{v}w\right)\left(p\right):=\mathrm{G}_{u}\left(\mathrm{G}_{v}w\left(p\right)\right)\left(p\right)=\mathrm{G}_{u}\left(p^{-1}\circ\left(p+\overline{v}\right)\left(w\right)\right)\left(p\right)\label{eq:G}\\
=p^{-1}\circ\left(p+u\right)\circ p^{-1}\circ\left(p+\overline{v}\right)\left(w\right),\nonumber 
\end{gather}
\begin{align*}
\mathrm{D}_{u}\left(\mathrm{D}_{v}w\right)\left(p\right): & =\mathrm{D}_{u}\left(\mathrm{D}_{v}w\left(p\right)\right)\left(p\right)=\mathrm{D}_{u}\left(p^{-1}\circ\left(p+\overline{v}\right)\left(w\right)-w\right)\left(p\right)\\
 & =p^{-1}\circ\left(p+\overline{u}\right)\left(p^{-1}\circ\left(p+\overline{v}\right)\left(w\right)-w\right)-\left(p^{-1}\circ\left(p+\overline{v}\right)\left(w\right)-w\right)\\
 & =p^{-1}\circ\left(p+\overline{u}\right)\circ p^{-1}\circ\left(p+\overline{v}\right)\left(w\right)-p^{-1}\circ\left(p+\overline{u}\right)\left(w\right)-p^{-1}\circ\left(p+\overline{v}\right)\left(w\right)+w
\end{align*}
and so on for more than three vectors. Note that hence
\begin{equation}
\mathrm{D}_{u}\left(\mathrm{D}_{v}w\right)\left(p\right)=\mathrm{G}_{u}\left(\mathrm{G}_{v}w\right)\left(p\right)-\mathrm{G}_{u}w\left(p\right)-\mathrm{G}_{v}w\left(p\right)+w.\label{eq:dg}
\end{equation}

$\vec{\mathrm{G}}{}_{u}\left(\mathrm{G}_{u}w\right)\left(p\right)$,
$\vec{\mathrm{D}}_{u}\left(\mathrm{D}_{v}w\right)\left(p\right)$
etc. can be defined analogously. $\mathrm{G}_{u}\left(\mathrm{G}_{v}w\right)\left(p\right)$,
$\mathrm{D}_{u}\left(\mathrm{D}_{v}w\right)\left(p\right)$, $\vec{\mathrm{G}}{}_{u}\left(\mathrm{G}_{v}w\right)\left(p\right)$
and $\vec{\mathrm{D}}_{u}\left(\mathrm{D}_{v}w\right)\left(p\right)$
are clearly linear in $w$ for all $u,v\in V$ and $p\in A$. It also
follows from $\mathrm{D}_{0}v\left(p\right)=\mathrm{D}_{u}0\left(p\right)=0$
that $\mathrm{D}_{0}\left(\mathrm{D}_{v}w\right)\left(p\right)=\mathrm{D}_{u}\left(\mathrm{D}_{0}w\right)\left(p\right)=0$
etc.
\begin{defn}
\label{def:43}A\emph{ (point) displacement} is a function
\begin{gather*}
\mathrm{M}:V\times V\times A\rightarrow A,\\
\left(u,v,p\right)\mapsto\mathrm{M}_{u}v\left(p\right):=\left(p+\overline{u}\right)+\vec{\mathrm{G}}_{u}v\left(p\right)=\left(p+\overline{u}\right)+\left(p+\overline{u}\right)\left(v\right)=\left(p+\overline{u}\right)+\overline{v}.
\end{gather*}
\end{defn}

A (once) iterated displacement is given by
\begin{align}
\mathrm{M}_{uv}w\left(p\right) & =\left(\left(p+\overline{u}\right)+\overline{v}\right)+\vec{\mathrm{G}}_{u}\left(\mathrm{G}_{v}w\left(p\right)\right)\left(p\right)\label{eq:m2}\\
 & =\left(\left(p+\bar{u}\right)+\overline{v}\right)+\left(p+\overline{u}\right)\circ p^{-1}\circ\left(p+\bar{v}\right)\left(w\right).\nonumber 
\end{align}

$\mathrm{M}_{u}v\left(p\right)$ corresponds to an \emph{interval
displacement} \emph{of $p\cdots p+\overline{v}$ along $p\cdots p+\overline{u}$}
\[
p\cdots p+\overline{v}\mapsto\left(p+\overline{u}\right)\cdots\left(p+\overline{u}\right)+\overline{v},
\]
 or a\emph{ parallel transport of a bound vector $v_{p}$ along $u_{p}$}
\begin{alignat*}{1}
v_{p}\mapsto p^{-1}\left(\left(\left(p+\overline{u}\right)+\overline{v}\right)\leftarrow\left(p+\overline{u}\right)\right)_{\left(p+\overline{u}\right)} & =p^{-1}\circ\left(p+\overline{u}\right)\left(v\right)_{\left(p+\overline{u}\right)}=\mathrm{G}_{u}v\left(p\right)_{\left(p+\overline{u}\right)}.
\end{alignat*}
Thus, the deviation $\mathrm{G}$ is the key to the dissociation $\mathrm{D}$,
the point displacement $\mathrm{M}$, and the classical notion of
parallel transport of vectors.

\subsection{\label{subsec:42}Infinitesimal dissociation}

In proper differential geometry, we have notions of differentiable
manifolds and curvature. We can similarly introduce notions of differentiability
and infinitesimal curvature in a pointwise affine space. For this
purpose, we define an operator $\Delta$ in terms of the dissociation
$\mathrm{D}$ by forming a natural limit.
\begin{defn}
\label{def:11}A \emph{differentiable} pointwise affine space is a
pointwise affine space $A$ such that the limit
\[
\Delta_{u}v\left(p\right)=\lim_{\tau\rightarrow0}\frac{\mathrm{D}{}_{\left(\tau u\right)}v\left(p\right)}{\tau}=\lim_{\tau\rightarrow0}\frac{p^{-1}\circ\left(p+\overline{\tau u}\right)\left(v\right)-v}{\tau}
\]
exists for all $p\in A$ and $u,v\in V$, and is linear in $u$ and
$v$. 

We call $\Delta_{u}v\left(p\right)$ the \emph{infinitesimal dissociation
}of $v$ with respect to $u$ at $p,$ or \emph{the pseudo-derivative}
of $v$ with respect to $u$ at $p$.
\end{defn}

It follows from Definition \ref{def:11} that if $A$ is a differentiable
pointwise affine space then 
\begin{equation}
\lim_{\tau\rightarrow0}p^{-1}\circ\left(p+\overline{\tau u}\right)\left(v\right)-v=0\label{eq:cont}
\end{equation}
 for all $u,v\in V$ and $p\in A$.

It is obvious how to extend Definition \ref{def:42} to the infinitesimal
case.
\begin{defn}
\label{def:45}Consider a pointwise affine system $\mathscr{A}\left(V,A,\boldsymbol{\mathfrak{a}}\right)$.
The function 
\[
C\left(p\right):\left(u,v\right)\mapsto\Delta{}_{u}v\left(p\right)
\]
 is the (infinitesimal)\emph{ affine curvature} of $A$ at $p$, and
a function 
\[
\boldsymbol{\mathit{C}}:p\mapsto C\left(p\right)
\]
 is an (infinitesimal)\emph{ affine connection} on $A$.
\end{defn}

\begin{prop}
\label{prop:21}If $\Delta_{u}v\left(p\right)$ is additive in $u$
then $\Delta_{u}v\left(p\right)$ is linear in $u$ and $v$.
\end{prop}

\begin{proof}
It suffices to note that
\begin{align*}
\Delta_{cu}v\left(p\right) & =\lim_{\tau\rightarrow0}\frac{p^{-1}\circ\left(p+\overline{\tau\left(cu\right)}\right)\left(v\right)-v}{\tau}=c\lim_{\tau\rightarrow0}\frac{p^{-1}\circ\left(p+\overline{\left(c\tau\right)u}\right)\left(v\right)-v}{c\tau}=c\,\Delta_{u}v\left(p\right)
\end{align*}
for all $c\neq0$, and that $\Delta_{cu}v\left(p\right)=c\,\Delta_{u}v\left(p\right)=0$
if $c=0$.
\end{proof}
A twice differentiable pointwise affine space is a differentiable
pointwise affine space such that the limit 
\[
\mathrm{\Delta}_{u}\left(\mathrm{\Delta}_{v}w\right)\left(p\right)=\lim_{\tau\rightarrow0}\frac{\mathrm{\mathrm{D}}{}_{\tau u}\left(\Delta_{v}w\left(p\right)\right)\left(p\right)}{\tau}=\lim_{\tau\rightarrow0}\frac{\mathrm{\mathrm{D}}{}_{\tau u}\left(\lim_{\tau\rightarrow0}\frac{\mathrm{D}{}_{\tau v}w\left(p\right)}{\tau}\right)\left(p\right)}{\tau}
\]
exists and is linear in $u,v,w$, and so on for three or more times
differentiable pointwise affine spaces.%

Note that differentiability of a pointwise affine space is an intrinsic
property of the space, independent of differentiability of functions.
Differentiation of functions will be considered in Section \ref{sec:5}.

\section{Skew-curvature}

Below, we will again consider a pointwise affine system $\mathscr{A}=\left(V,A,\boldsymbol{\mathfrak{a}}\right)$.
The functions $\mathrm{D}_{u}v$ and $\Delta{}_{u}v$ discussed above
account for the affine curvature of $A$. The ''tonal amount'' of
curvature can be subdivided into components in various ways. In this
section, we will consider two curvature components frequently encountered
in the literature, obtained by forming the difference $\mathrm{D}_{u}v-\mathrm{D}_{v}u$
(torsion) anf the difference $\mathrm{D}_{u}\left(\mathrm{D}_{v}w\right)-\mathrm{D}_{v}\left(\mathrm{D}_{u}w\right)$
(affine Riemann curvature), as well as their infinitesimal analogues.
A combination of torsion and affine Riemann curvature will also be
discussed.

\subsection{Torsion (first-order affine skew-curvature)}
\begin{defn}
\label{d17-1}The \emph{discrete torsion $\mathrm{T}_{uv}\left(p\right)$
at} $p$ \emph{for} $u,v$ is given by the function
\begin{gather*}
\mathrm{T}:V\times V\times A\rightarrow V,\\
\left(u,v,p\right)\mapsto\mathrm{T}_{uv}\left(p\right):=\mathrm{D}_{u}v\left(p\right)-\mathrm{D}_{v}u\left(p\right)=\left(p^{-1}\circ\left(p+\overline{u}\right)\left(v\right)-v\right)-\left(p^{-1}\circ\left(p+\overline{v}\right)\left(u\right)-u\right).
\end{gather*}
A variant of $\mathrm{T}$ is $\vec{\mathrm{T}}$, given by
\[
\vec{\mathrm{T}}_{uv}\left(p\right):=p\left(\mathrm{T}_{uv}\left(p\right)\right)=\vec{\mathrm{D}}_{u}v\left(p\right)-\vec{\mathrm{D}}_{v}u\left(p\right)=\left(\left(p+\overline{u}\right)\left(v\right)-p\left(v\right)\right)-\left(\left(p+\overline{v}\right)\left(u\right)-p\left(u\right)\right).
\]
\end{defn}

Clearly, $\mathrm{T}_{uv}\left(p\right)+\mathrm{T}_{vu}\left(p\right)=0$
and $\vec{\mathrm{T}}_{uv}\left(p\right)+\vec{\mathrm{T}}_{vu}\left(p\right)=\vec{0}$.
Furthermore, it follows from $\mathrm{D}_{0}u\left(p\right)=\mathrm{D}_{u}0\left(p\right)=0$
that $\mathrm{T}_{0v}\left(p\right)=\mathrm{T}_{u0}\left(p\right)=0$
and $\vec{\mathrm{T}}_{0v}\left(p\right)=\mathrm{\vec{\mathrm{T}}}_{u0}\left(p\right)=\vec{0}$
for all $u,v\in V$ and $p\in A$. 
\begin{prop}
\label{prop:15}$\vec{\mathrm{T}}_{uv}\left(p\right)=\left[\mathrm{M}_{u}v\left(p\right)\right]$,
where 
\[
\left[\mathrm{M}_{u}v\left(p\right)\right]=\mathrm{M}_{u}v\left(p\right)\leftarrow\mathrm{M}_{v}u\left(p\right)=\left(\left(p+\overline{u}\right)+\overline{v}\right)\leftarrow\left(\left(p+\overline{v}\right)+\overline{u}\right).
\]
\end{prop}

\begin{proof}
We have
\begin{align*}
\vec{\mathrm{T}}{}_{uv}\left(p\right) & =\left(\left(p+\overline{u}\right)\left(v\right)-p\left(v\right)\right)-\left(\left(p+\overline{v}\right)\left(u\right)-p\left(u\right)\right)\\
 & =\left(p\left(u\right)+\left(p+\overline{u}\right)\left(v\right)\right)-\left(p\left(v\right)+\left(p+\overline{v}\right)\left(u\right)\right)\\
 & =\left(p+\left(p\left(u\right)+\left(p+\overline{u}\right)\left(v\right)\right)\right)\leftarrow\left(p+\left(p\left(v\right)+\left(p+\overline{v}\right)\left(u\right)\right)\right)\\
 & =\left(\left(p+\overline{u}\right)+\overline{v}\right)\leftarrow\left(\left(p+\overline{v}\right)+\overline{u}\right).\qedhere
\end{align*}
As shown in Figure 5.1, Proposition \ref{prop:15} provides a simple
geometric interpretation of torsion, motivating Definition \ref{d17-1}.
\end{proof}
\begin{figure}[h]
\centering
\includegraphics[scale=0.15]{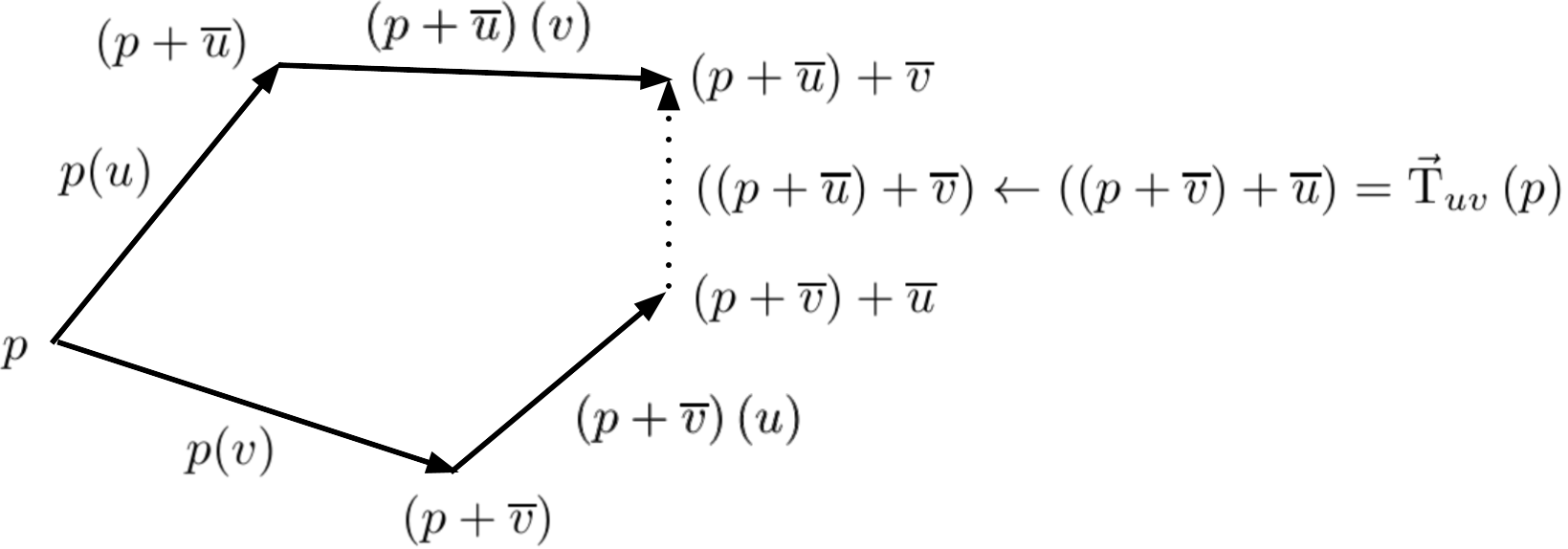}

\caption{Geometric interpretation of torsion. Points and translations in a
pointwise affine space.}

\end{figure}

Let us also define the infinitesimal torsion corresponding to $\mathrm{T}_{uv}\left(p\right)$.
\begin{defn}
\label{def:53}The\emph{ (infinitesimal) torsion $T_{uv}\left(p\right)$
at} $p$ \emph{for} $u,v$ is given by the function
\begin{gather*}
T:\mathcal{A}\times V\times V\rightarrow V,\qquad\left(u,v,p\right)\mapsto T_{uv}\left(p\right):=\Delta_{u}v\left(p\right)-\Delta_{v}u\left(p\right).
\end{gather*}
$T_{uv}\left(p\right)$ inherits anti-symmetry from $\mathrm{T}_{uv}\left(p\right)$
and is clearly linear in $u$ and $v$ as is $\Delta_{u}v$. In other
words, $\mathrm{T}_{uv}$ is a bilinear function $V\times V\rightarrow V$,
or a $\left(2,1\right)$-tensor.
\end{defn}

\subsection{Affine Riemann curvature (second-order affine skew-curvature)}
\begin{defn}
\label{s23-1} The \emph{discrete affine Riemann curvature $\mathrm{R}_{uvw}\left(p\right)$
at} $p$ \emph{for} $u,v.w$ is given by the function
\begin{gather*}
\mathrm{R}:V\times V\times V\times A\rightarrow V,\\
\left(u,v,w,p\right)\mapsto\mathrm{R}_{uvw}\left(p\right):=\mathrm{D}_{u}\left(\mathrm{D}_{v}w\right)\left(p\right)-\mathrm{D}_{v}\left(\mathrm{D}_{u}w\right)\left(p\right)
\end{gather*}
A variant of $\mathrm{R}$ is $\vec{\mathrm{R}}$, given by
\[
\vec{\mathrm{R}}{}_{uvw}\left(p\right)=p\left(\mathrm{R}_{uvw}\left(p\right)\right)=\vec{\mathrm{D}}{}_{u}\left(\mathrm{D}_{v}w\right)\left(p\right)-\vec{\mathrm{D}}_{v}\left(\mathrm{D}_{u}w\right)\left(p\right).
\]
 
\end{defn}

Clearly, $\mathrm{R}_{uvw}\left(p\right)+\mathrm{R}_{vuw}\left(p\right)=0$
for all $u,v,w\in V$, and similarly for $\vec{\mathrm{R}}$. It follows
from $\mathrm{D}_{0}u\left(p\right)=\mathrm{D}_{u}0\left(p\right)=0$
that $\mathrm{R}{}_{0vw}\left(p\right)=\mathrm{R}{}_{u0w}\left(p\right)=\mathrm{R}{}_{uv0}\left(p\right)=0$
for all $u,v,w\in V$ and $p\in A$, and there are analogous identities
for $\vec{\mathrm{R}}$.
\begin{lem}
\label{p24-1}$\mathrm{R}_{uvw}\left(p\right)=\mathrm{G}_{u}\left(\mathrm{G}_{v}w\right)\left(p\right)-\mathrm{G}_{v}\left(\mathrm{G}_{u}w\right)\left(p\right).$
\end{lem}

\begin{proof}
Using (\ref{eq:dg}), we obtain
\begin{align*}
\mathrm{R}_{uvw}\left(p\right) & =\mathrm{D}_{u}\left(\mathrm{D}_{v}w\right)\left(p\right)-\mathrm{D}_{v}\left(\mathrm{D}_{u}w\right)\left(p\right)\\
 & =\mathrm{G}_{u}\left(\mathrm{G}_{v}w\right)\left(p\right)-\mathrm{G}_{u}w\left(p\right)-\mathrm{G}_{v}w\left(p\right)+w\\
 & -\left(\mathrm{G}_{v}\left(\mathrm{G}_{u}w\right)\left(p\right)-\mathrm{G}_{v}w\left(p\right)-\mathrm{G}_{u}w\left(p\right)+w\right)\\
 & =\mathrm{G}_{u}\left(\mathrm{G}_{v}w\right)\left(p\right)-\mathrm{G}_{v}\left(\mathrm{G}_{u}w\right)\left(p\right).\qedhere
\end{align*}
\end{proof}
\begin{lem}
\label{lem:l18}Let $A$ be a point space. Then 
\begin{equation}
\left(p\leftarrow q\right)-\left(s\leftarrow r\right)=\left(p\leftarrow s\right)-\left(q\leftarrow r\right)\label{eq:perm}
\end{equation}
 for any $p,q,r,s\in A$.
\end{lem}

\begin{proof}
According to (\ref{eq:23-1}), $\left(p\leftarrow q\right)+\left(q\leftarrow r\right)=\left(p\leftarrow r\right)$,
so
\[
\vec{0}=\left(p\leftarrow p\right)=\left(p\leftarrow q\right)+\left(q\leftarrow r\right)+\left(r\leftarrow s\right)+\left(s\leftarrow p\right)=\left(p\leftarrow q\right)+\left(q\leftarrow r\right)-\left(s\leftarrow r\right)-\left(p\leftarrow s\right).
\]
Rearranging terms, we obtain (\ref{eq:perm}).
\end{proof}
\begin{prop}
\label{p25-1}$\vec{\mathrm{R}}{}_{uvw}\left(p\right)=\left[\mathrm{M}_{uv}w\left(p\right)\right]-\left[\mathrm{M}_{u}v\left(p\right)\right]$,
where 
\[
\left[\mathrm{M}_{uv}w\left(p\right)\right]=\mathrm{M}_{uv}w\left(p\right)\leftarrow\mathrm{M}_{vu}w\left(p\right)\quad\mathit{and}\quad\left[\mathrm{M}_{u}v\left(p\right)\right]=\mathrm{M}_{u}v\left(p\right)\leftarrow\mathrm{M}_{v}u\left(p\right).
\]
\end{prop}

\begin{proof}
Using Definition \ref{def:43}, the formula (\ref{eq:m2}) and Lemma
\ref{lem:l18}, we obtain 
\begin{align*}
\left[\mathrm{M}_{uv}w\left(p\right)\right]-\left[\mathrm{M}_{u}v\left(p\right)\right] & =\left(\mathrm{M}_{uv}w\left(p\right)\leftarrow\mathrm{M}_{vu}w\left(p\right)\right)-\left(\mathrm{M}_{u}v\left(p\right)\leftarrow\mathrm{M}_{v}u\left(p\right)\right)\\
 & =\left(\mathrm{M}_{uv}w\left(p\right)\leftarrow\mathrm{M}_{u}v\left(p\right)\right)-\left(\mathrm{M}_{vu}w\left(p\right)\leftarrow\mathrm{M}_{v}u\left(p\right)\right)\\
 & =\left(\left(\left(\left(p+\bar{u}\right)+\overline{v}\right)+\vec{\mathrm{G}}_{u}\left(\mathrm{G}_{v}w\left(p\right)\right)\left(p\right)\right)\leftarrow\left(\left(p+\bar{u}\right)+\overline{v}\right)\right)\\
 & -\left(\left(\left(\left(p+\bar{v}\right)+\overline{u}\right)+\vec{\mathrm{G}}_{v}\left(\mathrm{G}_{u}w\left(p\right)\right)\left(p\right)\right)\leftarrow\left(\left(p+\bar{v}\right)+\overline{u}\right)\right)\\
 & =\vec{\mathrm{G}}_{u}\left(\mathrm{G}_{v}w\left(p\right)\right)\left(p\right)-\vec{\mathrm{G}}_{v}\left(\mathrm{G}_{u}w\left(p\right)\right)\left(p\right)=p\left(\mathrm{G}_{u}\left(\mathrm{G}_{v}w\right)\left(p\right)-\mathrm{G}_{v}\left(\mathrm{G}_{u}w\right)\left(p\right)\right),
\end{align*}
so the assertion follows from Lemma \ref{p24-1}.
\end{proof}
\begin{figure}[h]
\includegraphics[scale=0.15]{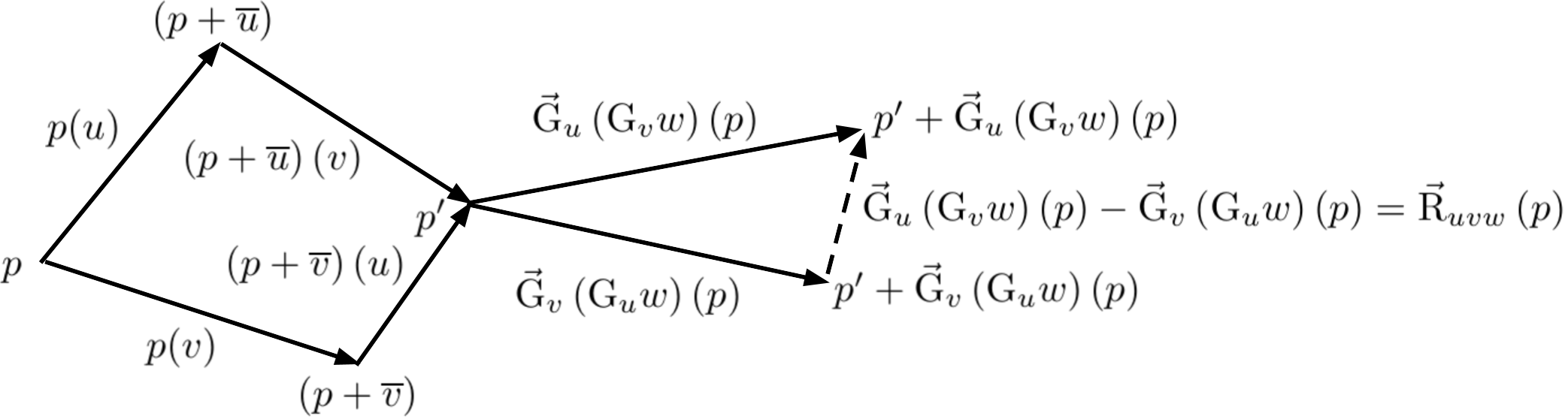}

\caption{Geometric interpretation of affine Riemann curvature in a torsion-free
context. Points and translations in a pointwise affine space. Recall
(eq. \ref{eq:G}) that $\vec{\mathrm{G}}_{u}\left(\mathrm{G}_{v}w\right)\left(p\right)=\left(p+\overline{u}\right)\circ p^{-1}\circ\left(p+\bar{v}\right)\left(w\right)$
etc.}
\end{figure}

As evident from Figure 5.2, Proposition \ref{p25-1} provides a geometric
interpretation of affine Riemann curvature in the special case when
$\mathrm{T}_{uv}\left(p\right)=0$ so that $\left[\mathrm{M}_{u}v\left(p\right)\right]=\vec{0}$.
This observation helps to motivate Definition \ref{s23-1}.

Finally, we define an infinitesimal affine Riemann curvature. Note
that this notion is more general than the usual metric Riemann curvature,
sometimes denoted $\mathring{R}_{uvw}\left(p\right)$, which is required
to be compatible with a given metric $g_{uv}$.
\begin{defn}
The \emph{(infinitesimal) affine Riemann curvature $R_{uvw}\left(p\right)$
at} $p$ \emph{for} $u,v,w$ is given by the function
\begin{gather*}
R:V\times V\times V\times A\rightarrow V,\qquad\left(u,v,w,p\right)\mapsto R_{uvw}\left(p\right):=\mathrm{\Delta}_{u}\left(\mathrm{\Delta}_{u}w\right)\left(p\right)-\mathrm{\Delta}_{u}\left(\mathrm{\Delta}_{u}w\right)\left(p\right).
\end{gather*}
\end{defn}

$R_{uvw}\left(p\right)$ inherits anti-symmetry with respect to $u,v$
from $\mathrm{R}{}_{uvw}\left(p\right)$ and is clearly linear in
$u,v,w$. In other words, $\mathrm{R}{}_{uvw}$ is a trilinear function
$V\times V\times V\rightarrow V$, or a $\left(3,1\right)$-tensor.

\subsection{Cumulative second-order affine skew-curvature}

We can combine the two types of curvature considered above.
\begin{defn}
\label{def:d20}The \emph{discrete affine cumulative second-order
curvature $\mathrm{C}{}_{uvw}\left(p\right)$ at} $p$ \emph{for}
$u,v.w$ is given by the function
\begin{gather*}
\mathrm{C}:V\times V\times V\times A\rightarrow V,\\
\left(u,v,w,p\right)\mapsto\mathrm{C}_{uvw}\left(p\right):=\mathrm{T}_{uv}\left(p\right)+\mathrm{R}_{uvw}\left(p\right).
\end{gather*}
\end{defn}

A variant of $\mathrm{C}$ is $\vec{\mathrm{C}}$, given by
\[
\vec{\mathrm{C}}_{uv}w\left(p\right):=p\left(\mathrm{C}_{uvw}\left(p\right)\right)=\vec{\mathrm{T}}{}_{u}u\left(p\right)+\vec{\mathrm{R}}{}_{uvw}\left(p\right).
\]

\begin{prop}
\label{prop:Cint}$\vec{\mathrm{C}}_{uvw}\left(p\right)=\left[\mathrm{M}_{uv}w\left(p\right)\right]$.
\end{prop}

\begin{proof}
Immediate from Propositions \ref{prop:15} and \ref{p25-1}.
\end{proof}
Clearly, $\mathrm{C}{}_{uvw}\left(p\right)+\mathrm{C}{}_{vuw}\left(p\right)=0$
for all $u,v,w\in V$, and similarly for $\vec{\mathrm{C}}$. It follows
from $\mathrm{D}_{0}u\left(p\right)=\mathrm{D}_{u}0\left(p\right)=0$
that $\mathrm{C}{}_{0vw}\left(p\right)=\mathrm{C}{}_{u0w}\left(p\right)=\mathrm{C}{}_{uv0}\left(p\right)=0$
for all $u,v,w\in V$ and $p\in\mathcal{A}$, and similarly for $\vec{\mathrm{C}}$.

In view of Proposition \ref{prop:Cint}, Figure 5.3 gives a geometric
interpretation of the relation between $\mathrm{T}_{uv}$, $\mathrm{R}_{uvw}$
and $\mathrm{C}_{uvw}$ or their infinitesimal analogues.

\begin{figure}[h]
\includegraphics[scale=0.15]{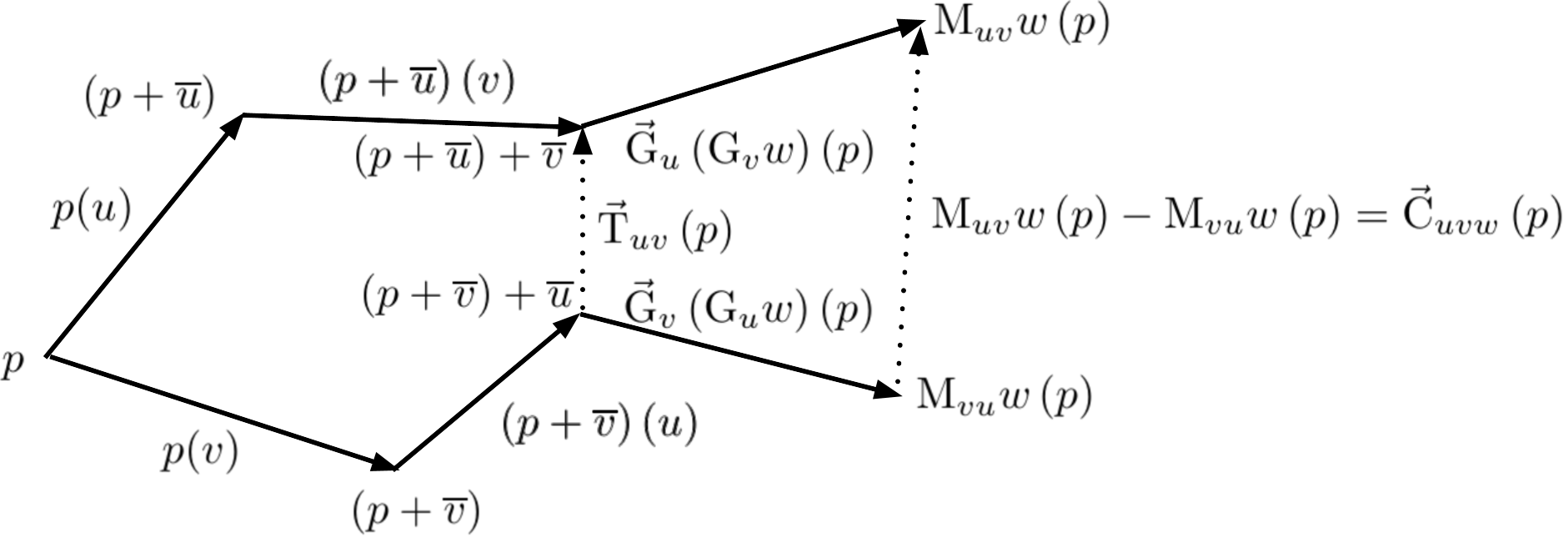}

\caption{Geometric interpretation of second order skew-curvature as sum of
torsion and affine Riemann curvature. Points and translations in a
pointwise affine space. Recall that $\vec{\mathrm{G}}_{u}\left(\mathrm{G}_{v}w\right)\left(p\right)=\left(p+\overline{u}\right)\circ p^{-1}\circ\left(p+\bar{v}\right)\left(w\right)$
and $\mathrm{M}_{u}v\left(p\right)=\left(\left(p+\overline{u}\right)+\overline{v}\right)+\vec{\mathrm{G}}_{u}\left(\mathrm{G}_{v}w\right)\left(p\right)$
etc.}
 
\end{figure}
By Figure 5.3, $\vec{\mathrm{T}}_{uv}\left(p\right)+\mathrm{\vec{G}}_{u}\left(\mathrm{G}_{v}w\right)\left(p\right)-\mathrm{\vec{C}}_{uvw}\left(p\right)-\mathrm{\vec{G}}_{v}\left(\mathrm{G}_{u}w\right)\left(p\right)=\vec{0}$,
so $\mathrm{\vec{C}}_{uvw}\left(p\right)-\vec{\mathrm{T}}_{uv}\left(p\right)=\mathrm{\vec{G}}_{u}\left(\mathrm{G}_{v}w\right)\left(p\right)-\mathrm{\vec{G}}_{v}\left(\mathrm{G}_{u}w\right)\left(p\right)$
and thus
\[
\mathrm{R}_{uvw}\left(p\right)=\mathrm{D}_{u}\left(\mathrm{D}_{v}w\right)\left(p\right)-\mathrm{D}_{v}\left(\mathrm{D}_{u}w\right)\left(p\right)=\mathrm{G}_{u}\left(\mathrm{G}_{v}w\right)\left(p\right)-\mathrm{G}_{v}\left(\mathrm{G}_{u}w\right)\left(p\right)=\mathrm{C}_{uvw}\left(p\right)-\mathrm{T}_{uv}\left(p\right).
\]

\begin{defn}
The \emph{(infinitesimal) cumulative second-order skew-curvature $C_{uvw}\left(p\right)$
at} $p$ \emph{for} $u,v.w$ is given by the function
\[
C:V\times V\times V\times A\rightarrow V,\qquad\left(u,v,w,p\right)\mapsto C{}_{uvw}\left(p\right):=T_{uv}\left(p\right)+R_{uvw}\left(p\right).
\]
 $C_{uvw}\left(p\right)$ inherits anti-symmetry with respect to $u,v$
from $\mathrm{C}{}_{uvw}\left(p\right)$ and is clearly linear in
$u,v,w$, and hence a $\left(3,1\right)$-tensor.%
\end{defn}

\section{\label{sec:5-1}Fields}

We have already encountered an example of a field, namely an affine
action field. In this article, an \emph{$S$-field on $A$} is simply
a function $f:A\rightarrow S$, where $A$ is a point space and $S$
is any set. We denote the set of all fields of the form $A\rightarrow S$
by $A^{S}$. 

If $S$ is a vector space $V$ over $K$ then $A^{V}$ is also a vector
space over $K$ with operations defined by
\[
\left(\varPhi+\varPsi\right)\left(p\right)=\varPhi\left(p\right)+\varPsi\left(p\right),\qquad\left(\lambda\varPhi\right)\left(p\right)=\lambda\left(\varPhi\left(p\right)\right).
\]
More generally, a set $A^{V}$ of functions closed under the displayed
operations is also a vector space. 

In the expression $\Delta_{u}v\left(p\right)$, $u$ and $v$ are
vectors in a given vector space $V$. We can replace these vectors
by vector fields $\boldsymbol{u},\boldsymbol{v}\in A^{V}$ by substituting
$\boldsymbol{u}\left(p\right)$ for $u$ and $\boldsymbol{v}\left(p\right)$
for $v$. Specifically, we set
\begin{gather*}
\Delta_{\boldsymbol{u}}v\left(p\right)=\left(\Delta_{\boldsymbol{u}\left(p\right)}v\right)\left(p\right)=\lim_{\tau\rightarrow0}\frac{p^{-1}\circ\left(p+\overline{\tau\boldsymbol{u}\left(p\right)}\right)\left(v\right)-v}{\tau},\\
\Delta_{u}\boldsymbol{v}\left(p\right)=\left(\Delta_{u}\boldsymbol{v}\left(p\right)\right)\left(p\right)=\lim_{\tau\rightarrow0}\frac{p^{-1}\circ\left(p+\overline{\tau u}\right)\left(\boldsymbol{v}\left(p\right)\right)-\boldsymbol{v}\left(p\right)}{\tau},\\
\Delta_{\boldsymbol{u}}\boldsymbol{v}\left(p\right)=\left(\Delta_{\boldsymbol{u}\left(p\right)}\boldsymbol{v}\left(p\right)\right)\left(p\right)=\lim_{\tau\rightarrow0}\frac{p^{-1}\circ\left(p+\overline{\tau\boldsymbol{u}\left(p\right)}\right)\left(\boldsymbol{v}\left(p\right)\right)-\boldsymbol{v}\left(p\right)}{\tau}
\end{gather*}
for all $p\in A$. Note that then $\Delta_{\boldsymbol{u}}v\left(p\right)$
and $\Delta_{\boldsymbol{u}}\boldsymbol{v}\left(p\right)$ do not
depend on the function $\boldsymbol{u}$ but only on its value $\boldsymbol{u}\left(p\right)$.

We can thus speak about an infinitesimal dissociation, or pseudo-derivative,
of a vector field $\boldsymbol{v}$, and along a vector field $\boldsymbol{u}$.
As a consequence, the vectors $u$, $v$ and $w$ in $T_{uv}$, $R_{uvw}$
and $C_{uvw}$ can be replaced by the vector fields $\boldsymbol{u}$,
$\boldsymbol{v}$ and $\boldsymbol{w}$.

A natural operation for fields on point spaces, especially vector
fields, is differentiation. In the derivative $\nabla_{u}\boldsymbol{v}\left(p\right)$,
discussed below, the vector field $\boldsymbol{v}$ cannot be replaced
by the vector $v$\emph{; }a derivative is a derivative of a field.
On the other hand, we can replace $u$ by the vector field $\boldsymbol{u}$,
setting
\[
\nabla_{\boldsymbol{u}}\boldsymbol{v}\left(p\right)=\nabla_{\boldsymbol{u}\left(p\right)}\boldsymbol{v}\left(p\right)
\]
 for all $p\in A.$ 

\section{\label{sec:5}Differentiation of fields}

Below, we will define several derivatives of point fields and vector
fields.%

\subsection{\label{subsec:51}Complete and reduced derivatives of point fields}

We first consider differentiation of mappings between point spaces
of pointwise affine systems $\mathscr{A}=\left(U,A,\boldsymbol{\mathfrak{a}}\right)$
and $\mathscr{B}=\left(V,B,\boldsymbol{\mathfrak{b}}\right)$. If
$\Phi:A\rightarrow B$ then we define a function 
\begin{gather*}
\dot{\overline{\mathsf{D}}}_{u}\left(\mathscr{A},\mathscr{B},\Phi\right):U\times A^{B}\times A\rightarrow V,\\
\left(u,\Phi,p\right)\mapsto{\textstyle \dot{\overline{\mathrm{D}}}}_{u}\Phi\left(p\right):=\boldsymbol{\mathfrak{b}}\left(\Phi\left(p\right)\right)^{-1}\left(\Phi\left(p+\boldsymbol{\mathfrak{a}}\left(p\right)\left(u\right)\right)\leftarrow\Phi\left(p\right)\right).
\end{gather*}

If $B\subseteq A$ then 
\[
\left\{ \Phi\left(p'\right)\leftarrow\Phi\left(p\right)\mid p,p'\in A\right\} \subseteq\vec{U}
\]
where $\vec{U}=\boldsymbol{\mathfrak{a}}\left(p\right)\left(U\right)$
for an arbitrary $p\in A$, so we can define a function
\begin{gather*}
\dot{\mathsf{D}}_{u}\left(\mathscr{A},\mathscr{B},\Phi\right):U\times A^{B}\times A\rightarrow U,\\
\left(u,\Phi,p\right)\mapsto\dot{\mathsf{D}}_{u}\Phi\left(p\right)=\boldsymbol{\mathfrak{a}}\left(p\right)^{-1}\left(\Phi\left(p+\boldsymbol{\mathfrak{a}}\left(p\right)\left(u\right)\right)\leftarrow\Phi\left(p\right)\right).
\end{gather*}
 In particular, this definition works if $\mathscr{B}=\mathscr{A}=\left(U,A,\boldsymbol{\mathfrak{a}}\right)$.

Using these two difference functions, we define the corresponding
derivatives as follows.
\begin{defn}
\label{def:24}Consider pointwise affine systems $\mathscr{A}=\left(U,A,\boldsymbol{\mathfrak{a}}\right)$
and $\mathscr{B}=\left(V,B,\boldsymbol{\mathfrak{b}}\right)$ with
an associated point field $\Phi:A\rightarrow B$. Assume that the
limits below exist and are linear in $u$.

The\emph{ complete derivative of $\Phi:A\rightarrow B$ with respect
to $u\in U$ at} $p\in A$ is the limit 
\begin{equation}
\dot{\overline{\nabla}}{}_{u}\Phi\left(p\right)=\lim_{\tau\rightarrow0}\frac{{\textstyle \dot{\overline{\mathsf{D}}}}_{\tau u}\Phi\left(p\right)}{\tau}=\lim_{\tau\rightarrow0}\frac{\boldsymbol{\mathfrak{b}}\left(\Phi\left(p\right)\right)^{-1}\left(\Phi\left(p+\overline{\tau u}\right)\leftarrow\Phi\left(p\right)\right)}{\tau},\label{eq:totder-1}
\end{equation}
where $\overline{\tau u}=\boldsymbol{\mathfrak{a}}\left(p\right)\left(\tau u\right)$.

If $B\subseteq A$ then we can define the\emph{ reduced derivative
of $\Phi:A\rightarrow B$ with respect to $u\in U$ at $p\in A$}
as the limit
\begin{equation}
\dot{\nabla}_{u}\Phi\left(p\right)=\lim_{\tau\rightarrow0}\frac{\dot{\mathsf{D}}_{\tau u}\Phi\left(p\right)}{\tau}=\lim_{\tau\rightarrow0}\frac{\boldsymbol{\mathfrak{a}}\left(p\right)^{-1}\left(\Phi\left(p+\overline{\tau u}\right)\leftarrow\Phi\left(p\right)\right)}{\tau}.\label{eq:totde-2}
\end{equation}
\end{defn}

Note that (\ref{eq:totder-1}) yields several special cases. For example,
if $\mathscr{B}=\left(V,B,\boldsymbol{\mathfrak{b}}\right)$ is in
effect an affine system since $\boldsymbol{\mathfrak{b}}\left(p\right)=\mathfrak{b}$
for all $p\in B$, then the\emph{ }derivative of $\Phi:A\rightarrow B$
with respect to $u\in U$ at $p\in A$ is 
\[
\dot{\overline{\nabla}}{}_{u}\Phi\left(p\right)=\lim_{\tau\rightarrow0}\frac{{\textstyle \dot{\overline{\mathsf{D}}}}_{\tau u}\Phi\left(p\right)}{\tau}=\lim_{\tau\rightarrow0}\frac{\mathfrak{b}^{-1}\left(\Phi\left(p+\overline{\tau u}\right)\leftarrow\Phi\left(p\right)\right)}{\tau}
\]

In this section, we have considered $\mathscr{B}=\left(V,B,\boldsymbol{\mathfrak{b}}\right)$
and used a function of one of the forms
\begin{gather*}
A\times A\rightarrow\vec{V},\qquad\left(p',p\right)\mapsto\left(\Phi\left(p'\right)\leftarrow\Phi\left(p\right)\right),\\
A\times A\rightarrow\vec{U},\qquad\left(p',p\right)\mapsto\left(\Phi\left(p'\right)\leftarrow\Phi\left(p\right)\right)
\end{gather*}
to define the derivative of a point field $\Phi:A\rightarrow B$.
Recall that $V$ can be regarded as an affine space, and in the next
two sections we will instead set $\mathscr{B}=\left(V,V,\boldsymbol{\mathfrak{b}}\right)$
and use a function of one of the forms
\begin{gather*}
A\times A\rightarrow\vec{V},\qquad\left(p',p\right)\mapsto\boldsymbol{\mathfrak{b}}\left(\varPhi\left(p'\right)\right)\left(\varPhi\left(p'\right)\right)-\boldsymbol{\mathfrak{b}}\left(\varPhi\left(p\right)\right)\left(\varPhi\left(p\right)\right),\\
A\times A\rightarrow\vec{U},\qquad\left(p',p\right)\mapsto\boldsymbol{\mathfrak{a}}\left(p'\right)\left(\varPhi\left(p'\right)\right)-\boldsymbol{\mathfrak{a}}\left(p\right)\left(\varPhi\left(p\right)\right)
\end{gather*}
to define the derivative of a vector field $\varPhi:A\rightarrow V$.

\subsection{\label{subsec:72}Complete and reduced derivatives of vector fields}

Let $U$ be a vector space over $K$ and consider the pointwise affine
systems $\mathscr{A}=\left(U,A,\boldsymbol{\mathfrak{a}}\right)$
and $\mathscr{B}=\left(V,V,\boldsymbol{\mathfrak{b}}\right)$.

In the general case, we define a function 
\begin{gather}
\overline{\mathsf{D}}_{u}\left(\mathscr{A},\mathscr{B},\varPhi\right):U\times A^{V}\times A\rightarrow V,\label{eq:pdiff1-2}\\
\left(u,\varPhi,p\right)\mapsto\overline{\mathsf{D}}_{u}\varPhi\left(p\right):=\boldsymbol{\mathfrak{b}}\left(\varPhi\left(p\right)\right)^{-1}\left(\boldsymbol{\mathfrak{b}}\left(\varPhi\left(p+\boldsymbol{\mathfrak{a}}\left(p\right)\left(u\right)\right)\right)\left(\varPhi\left(p+\boldsymbol{\mathfrak{a}}\left(p\right)\left(u\right)\right)\right)-\boldsymbol{\mathfrak{b}}\left(\varPhi\left(p\right)\right)\left(\varPhi\left(p\right)\right)\right)\nonumber \\
\qquad\qquad\quad\;=\boldsymbol{\mathfrak{b}}\left(\varPhi\left(p\right)\right)^{-1}\circ\boldsymbol{\mathfrak{b}}\left(\varPhi\left(p+\boldsymbol{\mathfrak{a}}\left(p\right)\left(u\right)\right)\right)\left(\varPhi\left(p+\boldsymbol{\mathfrak{a}}\left(p\right)\left(u\right)\right)\right)-\varPhi\left(p\right)\nonumber \\
=\boldsymbol{\mathfrak{b}}\left(\varPhi\left(p\right)\right)^{-1}\circ\boldsymbol{\mathfrak{b}}\left(\varPhi\left(p+\overline{u}\right)\right)\left(\varPhi\left(p+\overline{u}\right)\right)-\varPhi\left(p\right)
\end{gather}

If $V\leq U$ ($V$ is a subspace of $U)$ then
\begin{gather*}
\left\{ \varPhi\left(p\right)\mid p\in A\right\} \subseteq U.
\end{gather*}
so we can define a function

\begin{align}
\mathsf{D}_{u}\left(\mathscr{A},\mathscr{B},\varPhi\right) & :U\times A^{V}\times A\rightarrow U,\nonumber \\
\left(u,v,p\right)\mapsto\mathsf{D}_{u}\varPhi\left(p\right) & :=\boldsymbol{\mathfrak{a}}\left(p\right)^{-1}\left(\boldsymbol{\mathfrak{a}}\left(p+\boldsymbol{\mathfrak{a}}\left(p\right)\left(u\right)\right)\left(\varPhi\left(p+\boldsymbol{\mathfrak{a}}\left(p\right)\left(u\right)\right)\right)-\boldsymbol{\mathfrak{a}}\left(p\right)\left(\varPhi\left(p\right)\right)\right)\label{eq:26-1}\\
 & \,=p^{-1}\circ\left(p+\overline{u}\right)\left(\varPhi\left(p+\overline{u}\right)\right)-\varPhi\left(p\right).\nonumber 
\end{align}

Using these two difference functions, we obtain the following derivatives.
\begin{defn}
\label{def:d23}Consider the pointwise affine systems $\mathscr{A}=\left(U,A,\boldsymbol{\mathfrak{a}}\right)$
and $\mathscr{B}=\left(V,V,\boldsymbol{\mathfrak{b}}\right)$ with
an associated vector field $\varPhi:A\rightarrow V$, and write $\varPhi$
as $\boldsymbol{v}$. Assume in all cases that the limits exist and
are linear in $u$.

\emph{The complete derivative of $\boldsymbol{v}:A\rightarrow V$
with respect to $u\in U$ at $p\in A$} is
\begin{align}
\overline{\nabla}{}_{u}\boldsymbol{v}\left(p\right)=\lim_{\tau\rightarrow0}\frac{\overline{\mathsf{D}}_{\tau u}\varPhi\left(p\right)}{\tau} & =\lim_{\tau\rightarrow0}\frac{\boldsymbol{\mathfrak{b}}\left(\boldsymbol{v}\left(p\right)\right)^{-1}\circ\boldsymbol{\mathfrak{b}}\left(\boldsymbol{v}\left(p+\overline{\tau u}\right)\right)\left(\boldsymbol{v}\left(p+\overline{\tau u}\right)\right)-\boldsymbol{v}\left(p\right)}{\tau}.\label{eq:vm}
\end{align}

If $V\leq U$ then the \emph{reduced derivative of $\boldsymbol{v}$
with respect to $u\in U$ at $p\in A$} is
\begin{equation}
\nabla_{u}\boldsymbol{v}\left(p\right)=\lim_{\tau\rightarrow0}\frac{\mathsf{D}_{\tau u}\boldsymbol{v}\left(p\right)}{\tau}=\lim_{\tau\rightarrow0}\frac{p^{-1}\circ\left(p+\overline{\tau u}\right)\left(\boldsymbol{v}\left(p+\overline{\tau u}\right)\right)-\boldsymbol{v}\left(p\right)}{\tau}.\label{eq:d26-2}
\end{equation}
\end{defn}

It is clear from these definitions that $\overline{\nabla}{}_{u}\boldsymbol{v}$
and $\nabla_{u}\boldsymbol{v}$ are linear in $\boldsymbol{v}$ as
well as $u$.

\subsection{Plain derivatives to vector spaces }

Recall the definition of $\mathfrak{ci}_{V}$ in Section \ref{sec:Fundamental-notionss2}.
If $\boldsymbol{\mathfrak{b}}=\boldsymbol{\mathfrak{ci}}_{V}$, where
$\boldsymbol{\mathfrak{ci}}_{V}\left(x\right)=\mathfrak{ci}_{V}$
for all $x\in V$, then (\ref{eq:pdiff1-2}) yields
\begin{gather}
\tilde{\mathsf{D}}_{u}\left(\mathscr{A},\mathscr{B},\varPhi\right):U\times A^{V}\times A\rightarrow V,\label{eq:pfiff2-1}\\
\left(u,\varPhi,p\right)\mapsto\tilde{\mathsf{D}}_{u}\varPhi\left(p\right):=\boldsymbol{\mathfrak{ci}}_{V}\left(\varPhi\left(p\right)\right)^{-1}\left(\boldsymbol{\mathfrak{ci}}_{V}\left(\varPhi\left(p+\boldsymbol{\mathfrak{a}}\left(p\right)\left(u\right)\right)\right)\left(\varPhi\left(p+\boldsymbol{\mathfrak{a}}\left(p\right)\left(u\right)\right)\right)-\boldsymbol{\mathfrak{ci}}_{V}\left(\varPhi\left(p\right)\right)\left(\varPhi\left(p\right)\right)\right)\nonumber \\
=\mathfrak{ci}_{V}^{-1}\left(\mathfrak{ci}_{V}\left(\varPhi\left(p+\boldsymbol{\mathfrak{a}}\left(p\right)\left(u\right)\right)\right)-\mathfrak{ci}_{V}\left(\varPhi\left(p\right)\right)\right)=\varPhi\left(p+\boldsymbol{\mathfrak{a}}\left(p\right)u\right)-\varPhi\left(p\right)=\varPhi\left(p+\overline{u}\right)-\varPhi\left(p\right).\nonumber 
\end{gather}

\begin{defn}
The \emph{plain derivative to $V$ of $\boldsymbol{v}$ with respect
to $u\in U$ at $p\in A$} is
\begin{equation}
\delta{}_{u}\boldsymbol{v}\left(p\right)=\lim_{\tau\rightarrow0}\frac{\tilde{\mathsf{D}}{}_{\tau u}\boldsymbol{v}\left(p\right)}{\tau}=\lim_{\tau\rightarrow0}\frac{\boldsymbol{v}\left(p+\overline{\tau u}\right)-\boldsymbol{v}\left(p\right)}{\tau}.\label{eq:diffw}
\end{equation}
 It is clear from this definition that $\delta_{u}\boldsymbol{v}$
is linear in $\boldsymbol{v}$ as well as $u$.
\end{defn}

(\ref{eq:pfiff2-1}) is a special case of (\ref{eq:pdiff1-2}), so
a plain derivative to a vector space is a special case of a complete
derivative, but a common one. For example, if $\mathscr{A}=\left(U,A,\boldsymbol{\mathfrak{a}}\right)$
and $\mathscr{B}=\left(K,K,\boldsymbol{\mathfrak{ci}}{}_{K}\right)$
then the plain derivative to $K$ of the scalar function $\phi$ with
respect to $u\in U$ at $p\in A$ is given by
\begin{equation}
\delta{}_{u}\phi\left(p\right)=\lim_{\tau\rightarrow0}\frac{\tilde{\mathsf{D}}{}_{\tau u}\phi\left(p\right)}{\tau}=\lim_{\tau\rightarrow0}\frac{\phi\left(p+\overline{\tau u}\right)-\phi\left(p\right)}{\tau}.\label{eq:scalder}
\end{equation}
Similarly, if $\mathscr{A}=\left(U,A,\boldsymbol{\mathfrak{a}}\right)$
and $\mathscr{B}=\left(U^{*},U^{*},\boldsymbol{\mathfrak{ci}}{}_{U^{*}}\right)$
then the plain derivative to $U^{*}$ of the covector field $\varphi$
with respect to $u\in U$ at $p\in A$ is
\begin{equation}
\delta{}_{u}\boldsymbol{\varphi}\left(p\right)=\lim_{\tau\rightarrow0}\frac{\tilde{\mathsf{D}}{}_{\tau u}\varphi\left(p\right)}{\tau}=\lim_{\tau\rightarrow0}\frac{\boldsymbol{\varphi}\left(p+\overline{\tau u}\right)-\boldsymbol{\varphi}\left(p\right)}{\tau}\label{eq:covecDer-1}
\end{equation}
. Note that reduced derivatives do not exist in these cases, since
$U\neq K$ and $U\neq U{}^{*}$.

\subsection{\label{subsec:74}Plain derivatives on vector spaces}

In Section \ref{subsec:51}, we considered differentiation of functions
$f:A\rightarrow B$ between two point spaces, the most general setting;
in Section \ref{subsec:72} we considered functions of the form $f:A\rightarrow V$,
where $V$ is a vector space considered as a point space. So far,
$A$ was always a point space proper. In this subsection, we consider
differentiation of functions of the forms $f:U\mapsto B$ and $f:U\rightarrow V$,
where $U$ is a vector space regarded as a flat point space. 

Specifically, by setting $f\left(x+\overline{\tau u}\right)=f\left(x+\boldsymbol{\mathfrak{ci}}{}_{U}\left(x\right)\left(\tau u\right)\right)$
we obtain \emph{plain} \emph{derivatives on vector spaces}. Recall
that $\boldsymbol{\mathfrak{ci}}{}_{U}\left(x\right)\left(u\right)=\mathfrak{ci}_{U}\left(u\right)$
for all $x,u\in U$, so
\[
x+\overline{u}=x+x\left(u\right)=x+\mathfrak{ci}_{U}\left(u\right)=\mathfrak{ci}_{U}\left(u\right)\left(x\right)=u+x=x+u
\]
 for all $x,u\in U$. Thus, $f\left(x+\overline{\tau u}\right)=f\left(x+\tau u\right)$.
We introduce special notation for these derivatives.
\begin{enumerate}
\item[(a)] Corresponding to $\dot{\overline{\nabla}}{}_{u}\Phi\left(x\right)$
in Definition \ref{def:24} we have
\begin{equation}
{\textstyle \dot{\overline{\triangledown}}}{}_{u}\Phi\left(x\right)=\lim_{\tau\rightarrow0}\frac{{\textstyle {\textstyle \dot{\overline{\mathcal{D}}}}}_{\tau u}\Phi\left(p\right)}{\tau}=\lim_{\tau\rightarrow0}\frac{\boldsymbol{\mathfrak{b}}\left(\Phi\left(x\right)\right)^{-1}\left(\Phi\left(x+\tau u\right)\leftarrow\Phi\left(x\right)\right)}{\tau},\label{eq:p3-1}
\end{equation}
\item[(b)] Corresponding to $\dot{\nabla}{}_{u}\Phi\left(x\right)$ in Definition
\ref{def:24} we have
\begin{equation}
\dot{\triangledown}{}_{u}\Phi\left(x\right)=\lim_{\tau\rightarrow0}\frac{\dot{\mathcal{D}}_{\tau u}\Phi\left(p\right)}{\tau}=\lim_{\tau\rightarrow0}\frac{\mathbf{\boldsymbol{\mathfrak{a}}}\left(x\right)^{-1}\left(\Phi\left(x+\tau u\right)\leftarrow\Phi\left(x\right)\right)}{\tau},\label{eq:p3-1-1}
\end{equation}
\item[(c)] Corresponding to $\overline{\nabla}_{u}\boldsymbol{v}\left(x\right)$
in Definition \ref{def:d23} we have
\begin{equation}
\overline{\triangledown}{}_{u}\boldsymbol{v}\left(p\right)=\lim_{\tau\rightarrow0}\frac{\overline{\mathcal{D}}_{\tau u}\varPhi\left(p\right)}{\tau}=\lim_{\tau\rightarrow0}\frac{\boldsymbol{\mathfrak{b}}\left(\boldsymbol{v}\left(p\right)\right)^{-1}\circ\boldsymbol{\mathfrak{b}}\left(\boldsymbol{v}\left(p+\tau u\right)\right)\left(\boldsymbol{v}\left(p+\tau u\right)\right)-\boldsymbol{v}\left(p\right)}{\tau}.\label{eq:p2-1-1}
\end{equation}
\item[(d)] Corresponding to $\nabla_{u}\boldsymbol{v}\left(x\right)$ in Definition
\ref{def:d23} we have
\begin{equation}
\triangledown{}_{u}\boldsymbol{v}\left(x\right)=\lim_{\tau\rightarrow0}\frac{\mathcal{D}_{\tau u}\boldsymbol{v}\left(p\right)}{\tau}=\lim_{\tau\rightarrow0}\frac{x^{-1}\circ\left(x+\overline{\tau u}\right)\left(\boldsymbol{v}\left(x+\tau u\right)\right)-\boldsymbol{v}\left(x\right)}{\tau}.\label{eq:p2-1}
\end{equation}
\item[(e)] Corresponding to $\left(\delta{}_{u}\boldsymbol{v}\right)\left(x\right)$
in Definition \ref{def:d23} we have
\begin{equation}
\partial{}_{u}\boldsymbol{v}\left(x\right)=\lim_{\tau\rightarrow0}\frac{\tilde{\mathcal{D}}_{\tau u}\boldsymbol{v}\left(p\right)}{\tau}=\lim_{\tau\rightarrow0}\frac{\mathcal{\tilde{D}}{}_{\tau u}\boldsymbol{v}\left(p\right)}{\tau}=\lim_{\tau\rightarrow0}\frac{\boldsymbol{v}\left(x+\tau u\right)-\boldsymbol{v}\left(x\right)}{\tau},\label{eq:p1}
\end{equation}
corresponding to $\left(\delta{}_{u}\phi\right)\left(x\right)$ in
(\ref{eq:scalder}) we have
\begin{equation}
\partial{}_{u}\phi\left(x\right)=\lim_{\tau\rightarrow0}\frac{\mathcal{\tilde{D}}{}_{\tau u}\phi\left(p\right)}{\tau}=\lim_{\tau\rightarrow0}\frac{\phi\left(x+\tau u\right)-\phi\left(x\right)}{\tau},\label{eq:p1-1}
\end{equation}
 and corresponding to $\left(\delta{}_{u}\varphi\right)\left(x\right)$
in (\ref{eq:covecDer-1}) we have
\begin{equation}
\partial{}_{u}\varphi\left(x\right)=\lim_{\tau\rightarrow0}\frac{\mathcal{\tilde{D}}{}_{\tau u}\varphi\left(p\right)}{\tau}=\lim_{\tau\rightarrow0}\frac{\varphi\left(x+\tau u\right)-\phi\left(x\right)}{\tau}.\label{eq:p1-1-1}
\end{equation}
\end{enumerate}
Thus, we use the substitutions $\dot{\overline{\nabla}}_{u}\mapsto{\textstyle {\textstyle \dot{\overline{\triangledown}}}{}_{u}}$,
$\dot{\nabla}{}_{u}\mapsto\dot{\triangledown}{}_{u}$, $\overline{\nabla}_{u}\mapsto\overline{\triangledown}{}_{u}$,
$\nabla_{u}\mapsto\triangledown{}_{u}$ and $\delta_{u}\mapsto\partial_{u}$.

Note that the pointwise affine systems $\mathscr{A}$ associated with
complete derivatives on plain vector spaces (cases (a), (c) and (e))
have the form $\left(U,U,\boldsymbol{\mathfrak{ci}}_{U}\right)$,
while the pointwise affine systems $\mathscr{A}$ associated with
reduced derivatives on plain vector spaces (cases (b) and (d)) have
the form $\left(U,U,\boldsymbol{\mathfrak{a}}\right)$.

It is convenient to introduce notation that combines different kinds
of derivatives. We sometimes denote any and all of $\overline{\nabla}_{u},\overline{\triangledown}{}_{u},\nabla_{u},\triangledown_{u},\delta_{u},\partial_{u}$
by the generic derivative $\boldsymbol{\nabla}_{u}$. For example,
$\boldsymbol{\nabla}_{u}\lambda\boldsymbol{v}=\lambda\boldsymbol{\nabla}_{u}\boldsymbol{v}$
may be interpreted as $\nabla_{u}\lambda\boldsymbol{v}=\lambda\nabla_{u}\boldsymbol{v}$,
$\partial_{u}\lambda\boldsymbol{v}=\lambda\partial_{u}\boldsymbol{v}$
etc. A \emph{consistent} interpretation is one where where all occurrences
of $\boldsymbol{\nabla_{u}}$ are interpreted as the same derivative.
Some identities hold for many interpretations of $\boldsymbol{\nabla_{u}},$
often all consistent interpretations, so we can use this notation
as a shorthand.

\subsection{\label{subsec:75}Linearity properties}

Recall from Section \ref{subsec:42} that $\Delta_{u}v$ is linear
in the vectors $u$ and $v$,\emph{ }partly by assumption and partly
by calculation. The derivatives $\overline{\nabla}{}_{u}\boldsymbol{v}$,
$\nabla_{u}\boldsymbol{v}$ and $\delta{}_{u}\boldsymbol{v}$, and
variants thereof, have similar linearity properties. For example,
\begin{equation}
\nabla_{\lambda u+u'}\boldsymbol{v}\left(p\right)=\lambda\nabla_{u}\boldsymbol{v}\left(p\right)+\nabla_{u'}\boldsymbol{v}\left(p\right),\qquad\nabla_{u}\left(\lambda\boldsymbol{v}\left(p\right)+\boldsymbol{v}'\left(p\right)\right)=\lambda\nabla_{u}\boldsymbol{v}\left(p\right)+\nabla_{u}\boldsymbol{v}'\left(p\right)\label{eq:linder}
\end{equation}
 for any $\lambda\in K$, so $\nabla_{u}\boldsymbol{v}$ is linear
in $u$ and $\boldsymbol{v}$ with scalar multiplication by elements
of $K$.

Derivatives with regard to vector fields have analogous linearity
properties. Thus, if we set $\nabla_{\phi\boldsymbol{u}}\boldsymbol{v}\left(p\right)=\nabla_{\phi\left(p\right)\boldsymbol{u}\left(p\right)}\boldsymbol{v}\left(p\right)$
and $\phi\nabla_{\boldsymbol{u}}\boldsymbol{v}\left(p\right)=\phi\left(p\right)\nabla_{\boldsymbol{u}}\boldsymbol{v}\left(p\right)$
then (\ref{eq:linder}) generalizes to
\[
\nabla_{\left(\phi\boldsymbol{u}+\boldsymbol{u}'\right)}\boldsymbol{v}\left(p\right)=\phi\nabla_{\boldsymbol{u}}\boldsymbol{v}\left(p\right)+\nabla_{\boldsymbol{u}'}\boldsymbol{v}\left(p\right),\qquad\nabla_{\boldsymbol{u}}\left(\lambda\boldsymbol{v}\left(p\right)+\boldsymbol{v}'\left(p\right)\right)=\lambda\nabla_{\boldsymbol{u}}\boldsymbol{v}\left(p\right)+\nabla_{\boldsymbol{u}}\boldsymbol{v}'\left(p\right)
\]
for any $\phi\in A^{K}$ and $\lambda\in K$, so $\nabla_{\boldsymbol{u}}\boldsymbol{v}$
is linear in $\boldsymbol{u}$ with scalar multiplication by elements
of the ring $A^{K}$, and in $\boldsymbol{v}$ with scalar multiplication
by elements of $K$.

\section{Product rules}

There are well-known rules for differentiation of products of vector
fields. In this section, we formulate such rules for the somewhat
more general case of bilinear maps of vectors or vector fields.

Recall that if $\lim_{\tau\rightarrow a}X\left(\tau\right)$ exists
and $f:X\mapsto f\left(X\right)$ is continuous then $\lim_{\tau\rightarrow a}f\left(X\left(\tau\right)\right)=f\left(\lim_{\tau\rightarrow a}X\left(\tau\right)\right)$.
More generally, if $f:X,Y\mapsto f\left(X,Y\right)$ is continuous
and $\lim_{\tau\rightarrow a}X\left(\tau\right)$ and $\lim_{\tau\rightarrow a}Y\left(\tau\right)$
exist then 
\begin{equation}
\lim_{\tau\rightarrow a}f\left(X\left(\tau\right),Y\left(\tau\right)\right)=f\left(\lim_{\tau\rightarrow a}X\left(\tau\right),\lim_{\tau\rightarrow a}Y\left(\tau\right)\right).\label{eq:lim-lim}
\end{equation}

\subsection{Pseudo-derivatives}
\begin{prop}
Let $\left(v,w\right)\mapsto\left\langle \!\left\langle v,w\right\rangle \!\right\rangle $,
where $v$, $w$ and $\left\langle \!\left\langle v,w\right\rangle \!\right\rangle $
are vectors, be a bilinear map, and let $\Delta_{u}\left\langle \!\left\langle v,w\right\rangle \!\right\rangle \left(p\right)$
be defined by 
\[
\Delta_{u}\left\langle \!\left\langle v,w\right\rangle \!\right\rangle \left(p\right)=\lim_{\tau\rightarrow0}\frac{\left\langle \!\left\langle p^{-1}\circ\left(p+\overline{\tau u}\right)\left(v\right),p^{-1}\circ\left(p+\overline{\tau u}\right)\left(w\right)\right\rangle \!\right\rangle -\left\langle \!\left\langle v,w\right\rangle \!\right\rangle }{\tau}.
\]
If $\Delta{}_{u}v\left(p\right)$ and $\Delta{}_{u}w\left(p\right)$
exist then $\Delta_{u}\left\langle \!\left\langle v,w\right\rangle \!\right\rangle \left(p\right)$
exists and 
\[
\Delta_{u}\left\langle \!\left\langle v,w\right\rangle \!\right\rangle \left(p\right)=\left\langle \!\left\langle v,\Delta{}_{u}w\left(p\right)\right\rangle \!\right\rangle +\left\langle \!\left\langle \Delta{}_{u}v\left(p\right),w\right\rangle \!\right\rangle .
\]
\end{prop}

\begin{proof}
Set $P_{\tau}=p^{-1}\circ\left(p+\overline{\tau u}\right)$ to save
space. Using (\ref{eq:lim-lim}), we obtain
\begin{align*}
\Delta_{u}\left\langle \!\left\langle v,w\right\rangle \!\right\rangle \left(p\right) & =\lim\nolimits_{\tau\rightarrow0}\nicefrac{1}{\tau}\left(\left\langle \!\left\langle P_{\tau}\left(v\right),P_{\tau}\left(w\right)\right\rangle \!\right\rangle -\left\langle \!\left\langle v,w\right\rangle \!\right\rangle \right)\\
 & =\lim\nolimits_{\tau\rightarrow0}\nicefrac{1}{\tau}\left(\left\langle \!\left\langle P_{\tau}\left(v\right),P_{\tau}\left(w\right)\right\rangle \!\right\rangle -\left\langle \!\left\langle P_{\tau}\left(v\right),w\right\rangle \!\right\rangle +\left\langle \!\left\langle P_{\tau}\left(v\right),w\right\rangle \!\right\rangle -\left\langle \!\left\langle v,w\right\rangle \!\right\rangle \right)\\
 & =\lim\nolimits_{\tau\rightarrow0}\nicefrac{1}{\tau}\left(\left\langle \!\left\langle P_{\tau}\left(v\right),P_{\tau}\left(w\right)-w\right\rangle \!\right\rangle +\left\langle \!\left\langle P_{\tau}\left(v\right)-v,w\right\rangle \!\right\rangle \right)\\
 & =\lim\nolimits_{\tau\rightarrow0}\nicefrac{1}{\tau}\left\langle \!\left\langle P_{\tau}\left(v\right),P_{\tau}\left(w\right)-w\right\rangle \!\right\rangle +\lim\nolimits_{\tau\rightarrow0}\nicefrac{1}{\tau}\left\langle \!\left\langle P_{\tau}\left(v\right)-v,w\right\rangle \!\right\rangle \\
 & =\lim\nolimits_{\tau\rightarrow0}\left\langle \!\left\langle P_{\tau}\left(v\right),\nicefrac{1}{\tau}\left(P_{\tau}\left(w\right)-w\right)\right\rangle \!\right\rangle +\lim\nolimits_{\tau\rightarrow0}\left\langle \!\left\langle \nicefrac{1}{\tau}\left(P_{\tau}\left(v\right)-v\right),w\right\rangle \!\right\rangle \\
 & =\left\langle \!\left\langle \lim\nolimits_{\tau\rightarrow0}P_{\tau}\left(v\right),\lim\nolimits_{\tau\rightarrow0}\nicefrac{1}{\tau}\left(P_{\tau}\left(w\right)-w\right)\right\rangle \!\right\rangle +\left\langle \!\left\langle \lim\nolimits_{\tau\rightarrow0}\nicefrac{1}{\tau}\left(P_{\tau}\left(v\right)-v\right),w\right\rangle \!\right\rangle \\
 & =\left\langle \!\left\langle v,\Delta{}_{u}w\left(p\right)\right\rangle \!\right\rangle +\left\langle \!\left\langle \Delta{}_{u}v\left(p\right),w\right\rangle \!\right\rangle .\qedhere
\end{align*}
\end{proof}
Analogous results hold when vector fields are substituted for vectors
in $\Delta_{u}\left\langle \!\left\langle v,w\right\rangle \!\right\rangle \left(p\right)$.

\subsection{Coherent derivatives (uniform vector fields)}

If $\mathscr{A}=\left(U,A,\boldsymbol{\mathfrak{a}}\right)$, $\mathscr{B}=\left(V,V,\boldsymbol{\mathfrak{b}}\right)$
and $\mathscr{C}=\left(W,W,\boldsymbol{\mathfrak{c}}\right)$, where
$V,W\subseteq U$, then $\boldsymbol{v}:A\rightarrow V$ and $\boldsymbol{w}:A\rightarrow W$
are said to be uniform vector fields and a derivative of a bilinear
map $\left(\boldsymbol{v},\boldsymbol{w}\right)\mapsto\left\langle \!\left\langle \boldsymbol{v},\boldsymbol{w}\right\rangle \!\right\rangle $
is said to be a \emph{coherent} derivative.
\begin{lem}
\label{lem:29}If the pointwise affine space considered is differentiable
then 
\[
\lim\nolimits_{\tau\rightarrow0}p^{-1}\circ\left(p+\overline{\tau u}\right)\left(\boldsymbol{v}\left(p+\overline{\tau u}\right)\right)=\boldsymbol{v}\left(p\right).
\]
\end{lem}

\begin{proof}
Set $X\left(\tau\right)=p^{-1}\circ\left(p+\overline{\tau u}\right)\left(\boldsymbol{v}\left(p+\overline{\tau u}\right)\right)-\boldsymbol{v}\left(p+\overline{\tau u}\right)$
and $Y\left(\tau\right)=\boldsymbol{v}\left(p+\overline{\tau u}\right)$.
Using (\ref{eq:cont}), we obtain
\begin{align*}
\lim_{\tau\rightarrow0}p^{-1}\circ\left(p+\overline{\tau u}\right)\left(\boldsymbol{v}\left(p+\overline{\tau u}\right)\right) & =\lim_{\tau\rightarrow0}\left(X\left(\tau\right)+Y\left(\tau\right)\right)=\lim_{\tau\rightarrow0}X\left(\tau\right)+\lim_{\tau\rightarrow0}Y\left(\tau\right)\\
 & =\lim_{\tau\rightarrow0}Y\left(\tau\right)=\boldsymbol{v}\left(p\right).\qedhere
\end{align*}
\end{proof}
\begin{prop}
\label{prop:26a}Consider a differentiable pointwise affine space
$A$ and vector fields $\boldsymbol{\boldsymbol{v}},\boldsymbol{w}:A\rightarrow V$,
and let the reduced derivative $\nabla_{u}\left\langle \!\left\langle \boldsymbol{\boldsymbol{v}},\boldsymbol{w}\right\rangle \!\right\rangle \left(p\right)$
be defined by 
\[
\nabla_{u}\left\langle \!\left\langle \boldsymbol{v},\boldsymbol{w}\right\rangle \!\right\rangle \left(p\right)=\lim_{\tau\rightarrow0}\frac{\left\langle \!\left\langle p^{-1}\circ\left(p+\overline{\tau u}\right)\left(\boldsymbol{v}\left(p+\overline{\tau u}\right)\right),p^{-1}\circ\left(p+\overline{\tau u}\right)\left(\boldsymbol{w}\left(p+\overline{\tau u}\right)\right)\right\rangle \!\right\rangle -\left\langle \!\left\langle \boldsymbol{v}\left(p\right),\boldsymbol{w}\left(p\right)\right\rangle \!\right\rangle }{\tau}
\]
If the reduced derivatives $\nabla_{u}\boldsymbol{v}\left(p\right)$
and $\nabla_{u}\boldsymbol{w}\left(p\right)$ exist then $\nabla_{u}\left\langle \!\left\langle \boldsymbol{v},\boldsymbol{w}\right\rangle \!\right\rangle \left(p\right)$
exists and
\[
\nabla_{u}\left\langle \!\left\langle \boldsymbol{v},\boldsymbol{w}\right\rangle \!\right\rangle \left(p\right)=\left\langle \!\left\langle \boldsymbol{v}\left(p\right),\nabla_{u}\boldsymbol{w}\left(p\right)\right\rangle \!\right\rangle +\left\langle \!\left\langle \nabla_{u}\boldsymbol{v}\left(p\right),\boldsymbol{w}\left(p\right)\right\rangle \!\right\rangle .
\]
\end{prop}

\begin{proof}
Set $P_{\tau}=p^{-1}\circ\left(p+\overline{\tau u}\right)$. Using
(\ref{eq:lim-lim}) and Lemma \ref{lem:29}, we obtain
\begin{align*}
\nabla_{u}\left\langle \!\left\langle \boldsymbol{v},\boldsymbol{w}\right\rangle \!\right\rangle \left(p\right) & =\lim\nolimits_{\tau\rightarrow0}\nicefrac{1}{\tau}\left(\left\langle \!\left\langle P_{\tau}\left(\boldsymbol{v}\left(p+\overline{\tau u}\right)\right),P_{\tau}\left(\boldsymbol{w}\left(p+\overline{\tau u}\right)\right)\right\rangle \!\right\rangle -\left\langle \!\left\langle \boldsymbol{v}\left(p\right),\boldsymbol{w}\left(p\right)\right\rangle \!\right\rangle \right)\\
 & =\lim\nolimits_{\tau\rightarrow0}\nicefrac{1}{\tau}\left(\left\langle \!\left\langle P_{\tau}\left(\boldsymbol{v}\left(p+\overline{\tau u}\right)\right),P_{\tau}\left(\boldsymbol{w}\left(p+\overline{\tau u}\right)\right)\right\rangle \!\right\rangle -\left\langle \!\left\langle P_{\tau}\left(\boldsymbol{v}\left(p+\overline{\tau u}\right)\right),\boldsymbol{w}\left(p\right)\right\rangle \!\right\rangle \right.\\
 & \qquad\qquad\qquad\,\left.+\left\langle \!\left\langle P_{\tau}\left(\boldsymbol{v}\left(p+\overline{\tau u}\right)\right),\boldsymbol{w}\left(p\right)\right\rangle \!\right\rangle -\left\langle \!\left\langle \boldsymbol{v}\left(p\right),\boldsymbol{w}\left(p\right)\right\rangle \!\right\rangle \right)\\
 & =\lim\nolimits_{\tau\rightarrow0}\nicefrac{1}{\tau}\left\langle \!\left\langle P_{\tau}\left(\boldsymbol{v}\left(p+\overline{\tau u}\right)\right),P_{\tau}\left(\boldsymbol{w}\left(p+\overline{\tau u}\right)\right)-\boldsymbol{w}\left(p\right)\right\rangle \!\right\rangle \\
 & +\lim\nolimits_{\tau\rightarrow0}\nicefrac{1}{\tau}\left\langle \!\left\langle P_{\tau}\left(\boldsymbol{v}\left(p+\overline{\tau u}\right)-\boldsymbol{v}\left(p\right)\right),\boldsymbol{w}\left(p\right)\right\rangle \!\right\rangle \\
 & =\lim\nolimits_{\tau\rightarrow0}\left\langle \!\left\langle P_{\tau}\left(\boldsymbol{v}\left(p+\overline{\tau u}\right)\right),\nicefrac{1}{\tau}\left(P_{\tau}\left(\boldsymbol{w}\left(p+\overline{\tau u}\right)\right)-\boldsymbol{w}\left(p\right)\right)\right\rangle \!\right\rangle \\
 & +\lim\nolimits_{\tau\rightarrow0}\left\langle \!\left\langle \nicefrac{1}{\tau}\left(P_{\tau}\left(\boldsymbol{v}\left(p+\overline{\tau u}\right)-\boldsymbol{v}\left(p\right)\right)\right),\boldsymbol{w}\left(p\right)\right\rangle \!\right\rangle \\
 & =\left\langle \!\left\langle \lim\nolimits_{\tau\rightarrow0}P_{\tau}\left(\boldsymbol{v}\left(p+\overline{\tau u}\right)\right),\lim\nolimits_{\tau\rightarrow0}\nicefrac{1}{\tau}\left(P_{\tau}\left(\boldsymbol{w}\left(p+\overline{\tau u}\right)\right)-\boldsymbol{w}\left(p\right)\right)\right\rangle \!\right\rangle \\
 & +\left\langle \!\left\langle \lim\nolimits_{\tau\rightarrow0}\nicefrac{1}{\tau}\left(P_{\tau}\left(\boldsymbol{v}\left(p+\overline{\tau u}\right)-\boldsymbol{v}\left(p\right)\right)\right),\boldsymbol{w}\left(p\right)\right\rangle \!\right\rangle \\
 & =\left\langle \!\left\langle \boldsymbol{v}\left(p\right),\nabla_{u}\boldsymbol{w}\left(p\right)\right\rangle \!\right\rangle +\left\langle \!\left\langle \nabla_{u}\boldsymbol{v}\left(p\right),\boldsymbol{w}\left(p\right)\right\rangle \!\right\rangle .\qedhere
\end{align*}
\end{proof}
The following result for plain derivatives to vector fields can be
proved similarly.
\begin{prop}
\label{prop:p30-1}Let $A$ be differentiable, consider the vector
fields $\boldsymbol{\boldsymbol{v}},\boldsymbol{w}:A\rightarrow V$
and let the plain derivative to $V$ $\delta_{u}\left\langle \!\left\langle \boldsymbol{\boldsymbol{v}},\boldsymbol{w}\right\rangle \!\right\rangle \left(p\right)$
be defined by 
\[
\delta_{u}\left\langle \!\left\langle \boldsymbol{v},\boldsymbol{w}\right\rangle \!\right\rangle \left(p\right)=\lim_{\tau\rightarrow0}\frac{\left\langle \!\left\langle \left(\boldsymbol{v}\left(p+\overline{\tau u}\right)\right),\left(\boldsymbol{w}\left(p+\overline{\tau u}\right)\right)\right\rangle \!\right\rangle -\left\langle \!\left\langle \boldsymbol{v}\left(p\right),\boldsymbol{w}\left(p\right)\right\rangle \!\right\rangle }{\tau}.
\]
If the plain derivatives to $V$ $\delta_{u}\boldsymbol{v}\left(p\right)$
and $\delta_{u}\boldsymbol{w}\left(p\right)$ exist then $\delta_{u}\left\langle \!\left\langle \boldsymbol{v},\boldsymbol{w}\right\rangle \!\right\rangle \left(p\right)$
exists and 
\[
\delta_{u}\left\langle \!\left\langle \boldsymbol{v},\boldsymbol{w}\right\rangle \!\right\rangle \left(p\right)=\left\langle \!\left\langle \boldsymbol{v}\left(p\right),\delta_{u}\boldsymbol{w}\left(p\right)\right\rangle \!\right\rangle +\left\langle \!\left\langle \delta_{u}\boldsymbol{v}\left(p\right),\boldsymbol{w}\left(p\right)\right\rangle \!\right\rangle .
\]
\end{prop}

We note some concrete applications of Propositions \ref{prop:26a}
and \ref{prop:p30-1}.
\begin{itemize}
\item If $f$ and $g$ are scalar fields then $\left(f,g\right)\mapsto fg$,
where $fg\left(p\right)=f\left(p\right)g\left(p\right)$, is a bilinear
map, so if $\delta{}_{u}f$ and $\delta{}_{u}g$ exist then
\[
\delta{}_{u}\left(fg\right)\left(p\right)=f\left(\delta{}_{u}g\right)\left(p\right)+\left(\delta{}_{u}f\right)g\left(p\right).
\]
\item If $\boldsymbol{v}$ and $\boldsymbol{w}$ are vector fields then
the bilinear form $\left(\boldsymbol{v},\boldsymbol{w}\right)\mapsto\left\langle \boldsymbol{v},\boldsymbol{w}\right\rangle $,
where $\left\langle \boldsymbol{v},\boldsymbol{w}\right\rangle \left(p\right)=\left\langle \boldsymbol{v}\left(p\right),\boldsymbol{w}\left(p\right)\right\rangle $,
is a bilinear map, so if $\delta{}_{u}\boldsymbol{v}$ and $\delta{}_{u}\boldsymbol{w}$
exist then
\[
\delta{}_{u}\left\langle \boldsymbol{v},\boldsymbol{w}\right\rangle \left(p\right)=\left\langle \boldsymbol{v},\delta{}_{u}\boldsymbol{w}\right\rangle \left(p\right)+\left\langle \delta{}_{u}\boldsymbol{v},\boldsymbol{w}\right\rangle \left(p\right).
\]
\item If $\boldsymbol{v}$ and $\boldsymbol{w}$ are vector fields then
the cross product $\left(\boldsymbol{v},\boldsymbol{w}\right)\mapsto\boldsymbol{v}\times\boldsymbol{w}$,
where $\left(\boldsymbol{v}\times\boldsymbol{w}\right)\left(p\right)=\boldsymbol{v}\left(p\right)\times\boldsymbol{w}\left(p\right)$,
is a bilinear map, so if $\nabla_{u}\boldsymbol{v}$ and $\nabla_{u}\boldsymbol{w}$
exist then
\[
\nabla{}_{u}\left(\boldsymbol{v}\times\boldsymbol{w}\right)\left(p\right)=\left(\boldsymbol{v}\times\nabla_{u}\boldsymbol{w}\right)\left(p\right)+\left(\nabla_{u}\boldsymbol{v}\times\boldsymbol{w}\right)\left(p\right).
\]
\item If $\boldsymbol{v}$ and $\boldsymbol{w}$ are vector fields then
the tensor product $\left(\boldsymbol{v},\boldsymbol{w}\right)\mapsto\boldsymbol{v}\otimes\boldsymbol{w}$,
the exterior product $\left(\boldsymbol{v},\boldsymbol{w}\right)\mapsto\boldsymbol{v}\wedge\boldsymbol{w}$
and the Clifford product $\left(\boldsymbol{v},\boldsymbol{w}\right)\mapsto\boldsymbol{v}\boldsymbol{w}$
are bilinear maps, so if $\delta_{u}\boldsymbol{v}$ and $\delta_{u}\boldsymbol{w}$
exist then
\begin{gather*}
\delta_{u}\left(\boldsymbol{v}\otimes\boldsymbol{w}\right)\left(p\right)=\boldsymbol{v}\otimes\left(\delta_{u}\boldsymbol{w}\right)\left(p\right)+\left\langle \left(\delta_{u}\boldsymbol{v}\right)\otimes\boldsymbol{w}\right\rangle \left(p\right),\\
\delta_{u}\left(\boldsymbol{v}\wedge\boldsymbol{w}\right)\left(p\right)=\boldsymbol{v}\wedge\left(\delta_{u}\boldsymbol{w}\right)\left(p\right)+\left\langle \left(\delta_{u}\boldsymbol{v}\right)\wedge\boldsymbol{w}\right\rangle \left(p\right),\\
\delta_{u}\left(\boldsymbol{v}\boldsymbol{w}\right)\left(p\right)=\boldsymbol{v}\left(\delta_{u}\boldsymbol{w}\right)\left(p\right)+\left\langle \left(\delta_{u}\boldsymbol{v}\right)\boldsymbol{w}\right\rangle \left(p\right).
\end{gather*}
\end{itemize}

\subsection{Mixed derivatives (disparate vector fields)}

A product rule can even be applied to bilinear maps involving derivatives
of different kinds, using a \emph{mixed derivative} of disparate vector
fields, such as $\nabla\delta{}_{u}\left\langle \!\left\langle \boldsymbol{\varphi},\boldsymbol{v}\right\rangle \!\right\rangle $
below.
\begin{prop}
\label{prop:39}Consider a covector field $\varphi:A\rightarrow V^{*}$
and a vector field $\boldsymbol{v}:A\rightarrow V$, let 
\[
\left(\varphi,\boldsymbol{v}\right)\mapsto\left\langle \!\left\langle \varphi,\boldsymbol{v}\right\rangle \!\right\rangle 
\]
 be a bilinear map, and let the mixed derivative $\nabla\delta{}_{u}\left\langle \!\left\langle \boldsymbol{\varphi},\boldsymbol{v}\right\rangle \!\right\rangle \left(p\right)$
at $p$ be defined by
\[
\nabla\delta{}_{u}\left\langle \!\left\langle \varphi,\boldsymbol{v}\right\rangle \!\right\rangle \left(p\right)=\lim_{\tau\rightarrow0}\frac{\left\langle \!\left\langle \left(\varphi\left(p+\overline{\tau u}\right)\right),p^{-1}\circ\left(p+\overline{\tau u}\right)\left(\boldsymbol{v}\left(p+\overline{\tau u}\right)\right)\right\rangle \!\right\rangle -\left\langle \!\left\langle \varphi\left(p\right),\boldsymbol{v}\left(p\right)\right\rangle \!\right\rangle }{\tau}.
\]
 If the plain derivative to $V^{*}$ $\delta{}_{u}\varphi$ and the
reduced derivative $\nabla_{u}\boldsymbol{v}$ exist then $\nabla\partial_{u}\left\langle \!\left\langle \varphi,\boldsymbol{v}\right\rangle \!\right\rangle $
exists and
\[
\nabla\delta{}_{u}\left\langle \!\left\langle \varphi,\boldsymbol{v}\right\rangle \!\right\rangle \left(p\right)=\left\langle \!\left\langle \varphi,\nabla_{u}\boldsymbol{v}\right\rangle \!\right\rangle \left(p\right)+\left\langle \!\left\langle \delta{}_{u}\varphi,\boldsymbol{v}\right\rangle \!\right\rangle \left(p\right).
\]
\end{prop}

\begin{proof}
Set $P_{\tau}=p^{-1}\circ\left(p+\overline{\tau u}\right)$. Using
(\ref{eq:lim-lim}) and Lemma \ref{lem:29}, we have
\begin{align*}
\nabla\delta{}_{u}\left\langle \!\left\langle \varphi,\boldsymbol{v}\right\rangle \!\right\rangle \left(p\right) & =\lim\nolimits_{\tau\rightarrow0}\nicefrac{1}{\tau}\left(\left\langle \!\left\langle \left(\varphi\left(p+\overline{\tau u}\right)\right),P_{\tau}\left(\boldsymbol{v}\left(p+\overline{\tau u}\right)\right)\right\rangle \!\right\rangle -\left\langle \!\left\langle \varphi\left(p\right),\boldsymbol{v}\left(p\right)\right\rangle \!\right\rangle \right)\\
 & =\lim\nolimits_{\tau\rightarrow0}\nicefrac{1}{\tau}\left(\left\langle \!\left\langle \left(\varphi\left(p+\overline{\tau u}\right)\right),P_{\tau}\left(\boldsymbol{v}\left(p+\overline{\tau u}\right)\right)\right\rangle \!\right\rangle -\left\langle \!\left\langle \left(\varphi\left(p+\overline{\tau u}\right)\right),\boldsymbol{v}\left(p\right)\right\rangle \!\right\rangle \right.\\
 & \qquad\qquad\qquad\,\left.+\left\langle \!\left\langle \left(\varphi\left(p+\overline{\tau u}\right)\right),\boldsymbol{v}\left(p\right)\right\rangle \!\right\rangle -\left\langle \!\left\langle \varphi\left(p\right),\boldsymbol{v}\left(p\right)\right\rangle \!\right\rangle \right)\\
 & =\lim\nolimits_{\tau\rightarrow0}\nicefrac{1}{\tau}\left\langle \!\left\langle \left(\varphi\left(p+\overline{\tau u}\right)\right),P_{\tau}\left(\boldsymbol{v}\left(p+\overline{\tau u}\right)\right)-\boldsymbol{v}\left(p\right)\right\rangle \!\right\rangle \\
 & +\lim\nolimits_{\tau\rightarrow0}\nicefrac{1}{\tau}\left\langle \!\left\langle \left(\varphi\left(p+\overline{\tau u}\right)-\varphi\left(p\right)\right),\boldsymbol{v}\left(p\right)\right\rangle \!\right\rangle \\
 & =\lim\nolimits_{\tau\rightarrow0}\left\langle \!\left\langle \left(\varphi\left(p+\overline{\tau u}\right)\right),\nicefrac{1}{\tau}\left(P_{\tau}\left(\boldsymbol{v}\left(p+\overline{\tau u}\right)\right)-\boldsymbol{v}\left(p\right)\right)\right\rangle \!\right\rangle \\
 & +\lim\nolimits_{\tau\rightarrow0}\left\langle \!\left\langle \nicefrac{1}{\tau}\left(\left(\varphi\left(p+\overline{\tau u}\right)-\varphi\left(p\right)\right)\right),\boldsymbol{v}\left(p\right)\right\rangle \!\right\rangle \\
 & =\left\langle \!\left\langle \lim\nolimits_{\tau\rightarrow0}\left(\varphi\left(p+\overline{\tau u}\right)\right),\lim\nolimits_{\tau\rightarrow0}\nicefrac{1}{\tau}\left(P_{\tau}\left(\boldsymbol{v}\left(p+\overline{\tau u}\right)\right)-\boldsymbol{v}\left(p\right)\right)\right\rangle \!\right\rangle \\
 & +\left\langle \!\left\langle \lim\nolimits_{\tau\rightarrow0}\nicefrac{1}{\tau}\left(\left(\varphi\left(p+\overline{\tau u}\right)-\varphi\left(p\right)\right)\right),\boldsymbol{v}\left(p\right)\right\rangle \!\right\rangle \\
 & =\left(\left\langle \!\left\langle \varphi\left(p\right),\nabla_{u}\boldsymbol{v}\left(p\right)\right\rangle \!\right\rangle +\left\langle \!\left\langle \delta{}_{u}\varphi\left(p\right),\boldsymbol{v}\left(p\right)\right\rangle \!\right\rangle \right).\qedhere
\end{align*}
\end{proof}
\begin{cor}
\label{cor:40}We have 
\begin{equation}
\delta{}_{u}\left(\varphi\left(\boldsymbol{v}\right)\right)\left(p\right)=\varphi\left(\nabla_{u}\boldsymbol{v}\right)\left(p\right)+\left(\delta_{u}\varphi\right)\left(\boldsymbol{v}\right)\left(p\right).\label{eq:c40}
\end{equation}
\end{cor}

\begin{proof}
$\nabla\delta{}_{u}\left(\varphi\left(\boldsymbol{v}\right)\right)$,
$\varphi\left(\nabla_{u}\boldsymbol{v}\right)$ and $\left(\delta_{u}\varphi\right)\left(\boldsymbol{v}\right)$
give bilinear maps of the form $\left(\varphi,\boldsymbol{v}\right)\mapsto\left\langle \!\left\langle \varphi,\boldsymbol{v}\right\rangle \!\right\rangle $,
so Proposition \ref{prop:39} applies. Also, the field $\varphi\left(\boldsymbol{v}\right)$
reduces to a scalar function, so the mixed derivative $\nabla\delta{}_{u}\left(\varphi\left(\boldsymbol{v}\right)\right)$
reduces to $\delta{}_{u}\left(\varphi\left(\boldsymbol{v}\right)\right)$,
a plain derivative to $K$.
\end{proof}
The following result is proved in the same way as Proposition \ref{prop:39}.
\begin{prop}
Let $\varphi$ and $\left(\varphi,\boldsymbol{v}\right)\mapsto\left\langle \!\left\langle \varphi,\boldsymbol{v}\right\rangle \!\right\rangle $
be defined as in Proposition \ref{prop:39}, and let the mixed derivative
$\triangledown\partial\left\langle \!\left\langle \boldsymbol{\varphi},\boldsymbol{v}\right\rangle \!\right\rangle \left(p\right)$
at $p$ be defined by
\[
\triangledown\partial\left\langle \!\left\langle \boldsymbol{\varphi},\boldsymbol{v}\right\rangle \!\right\rangle \left(p\right)=\lim_{\tau\rightarrow0}\frac{\left\langle \!\left\langle \left(\varphi\left(p+\tau u\right)\right),p^{-1}\circ\left(p+\overline{\tau u}\right)\left(\boldsymbol{v}\left(p+\tau u\right)\right)\right\rangle \!\right\rangle -\left\langle \!\left\langle \varphi\left(p\right),\boldsymbol{v}\left(p\right)\right\rangle \!\right\rangle }{\tau}.
\]
 If the plain derivative to $V^{*}$ $\delta{}_{u}\varphi$ and the
reduced derivative $\nabla_{u}\boldsymbol{v}$ exist then $\triangledown\partial\left\langle \!\left\langle \boldsymbol{\varphi},\boldsymbol{v}\right\rangle \!\right\rangle $
exists and
\[
\triangledown\partial\left\langle \!\left\langle \boldsymbol{\varphi},\boldsymbol{v}\right\rangle \!\right\rangle \left(p\right)=\left\langle \!\left\langle \varphi,\triangledown{}_{u}\boldsymbol{v}\right\rangle \!\right\rangle \left(p\right)+\left\langle \!\left\langle \partial{}_{u}\varphi,\boldsymbol{v}\right\rangle \!\right\rangle \left(p\right).
\]
\end{prop}

\begin{cor}
\label{cor:40-1}We have 
\begin{equation}
\partial{}_{u}\left(\varphi\left(\boldsymbol{v}\right)\right)\left(p\right)=\varphi\left(\triangledown{}_{u}\boldsymbol{v}\right)\left(p\right)+\left(\partial{}_{u}\varphi\right)\left(\boldsymbol{v}\right)\left(p\right).\label{eq:c40-1}
\end{equation}
\end{cor}

Without proof, we give two more specific product rules for mixed derivatives.
\begin{prop}
\label{prop:87}If $\phi$ is a scalar field, $\boldsymbol{v}$ is
a vector field, and $\delta{}_{u}\phi$ and $\nabla_{u}\boldsymbol{v}$
exist then $\nabla\delta{}_{u}\left(\phi\boldsymbol{v}\right)\left(p\right)$
exists and 
\[
\nabla\delta{}_{u}\left(\phi\boldsymbol{v}\right)\left(p\right)=\nabla{}_{u}\left(\phi\boldsymbol{v}\right)\left(p\right)=\phi\left(\nabla_{u}\boldsymbol{v}\right)\left(p\right)+\left(\delta{}_{u}\phi\right)\boldsymbol{v}\left(p\right).
\]
\end{prop}

\begin{cor}
\label{prop:33}If $\phi$ is a scalar field, $\boldsymbol{v}$ is
a vector field, and $\partial{}_{u}\phi$ and $\triangledown_{u}\boldsymbol{v}$
exist then $\triangledown\partial_{u}\left(\phi\boldsymbol{v}\right)\left(p\right)$
exists and 
\[
\triangledown\partial_{u}\left(\phi\boldsymbol{v}\right)\left(p\right)=\triangledown{}_{u}\left(\phi\boldsymbol{v}\right)\left(p\right)=\phi\left(\triangledown_{u}\boldsymbol{v}\right)\left(p\right)+\left(\partial{}_{u}\phi\right)\boldsymbol{v}\left(p\right).
\]
\end{cor}

\section{Sample applications}

\subsection{The gradient}
\begin{prop}
Let $V$ be a vector space with an inner product (non-degenerate real
bilinear form) $\left\langle -,-\right\rangle $, and consider a scalar
function $\phi:V\rightarrow K$. Then there is a unique function
\[
\nabla\phi:V\rightarrow V,\qquad x\mapsto\nabla\phi\left(x\right)
\]
 such that $\left\langle \nabla\phi\left(x\right),u\right\rangle =\partial{}_{u}\phi\left(x\right)$
for all $u,x\in V$.
\end{prop}

\begin{proof}
(Existence). (i) If $u=0$ then $\left\langle v,u\right\rangle =0$
and $\partial{}_{u}\phi\left(x\right)=0$, so $\left\langle v,u\right\rangle =\partial{}_{u}\phi\left(x\right)$
for all $v\in V.$ (ii) If $\left\langle v,u\right\rangle =0$ for
all $v\in V$ then $u=0$ since $\left\langle -,-\right\rangle $
is non-degenerate. Equivalently, if $u\neq0$ then $\left\langle v,u\right\rangle \neq0$
for some $v\in V$, and using such a vector $v$ we have 
\[
\left\langle \frac{\partial{}_{u}\phi\left(x\right)}{\left\langle v,u\right\rangle }v,u\right\rangle =\partial{}_{u}\phi\left(x\right).
\]
(Uniqueness). If $\left\langle v,u\right\rangle =\left\langle v',u\right\rangle $
for all $u\in V$ so that $\left\langle v-v',u\right\rangle =0$ for
all $u\in V$ then $v=v'$ since $\left\langle -,-\right\rangle $
is non-degenerate.
\end{proof}
We call $\nabla\phi\left(x\right)$ the \emph{gradient} of $\phi$
at $x\in V$.

\subsection{Affine geodesics}
\begin{defn}
Let $V$ be a vector space over $\mathbb{R}$ and consider $\mathscr{A}=\left(\mathbb{R},\mathbb{R},\boldsymbol{\mathfrak{ci}}{}_{\mathbb{R}}\right)$
and $\mathscr{B}=\left(V,B,\boldsymbol{\mathfrak{b}}\right)$. Consider
a function $\boldsymbol{\gamma}:\left(a,b\right)\rightarrow B$ such
that there is a strictly monotone function $\xi:\mathbb{R}\rightarrow\left(a,b\right)$
and some $v\in V$ such that
\[
\dot{\overline{\triangledown}}{}_{1}\boldsymbol{\gamma}\circ\xi\left(x\right)=\lim_{\tau\rightarrow0}\frac{\boldsymbol{\mathfrak{b}}\left(\gamma\left(\xi\left(x\right)\right)\right)^{-1}\left(\gamma\left(\xi\left(x+\tau\right)\right)\leftarrow\gamma\left(\xi\left(x\right)\right)\right)}{\tau}=v
\]
for all $x\in\mathbb{R}$. We call $\overline{\gamma}=\gamma\circ\xi$
an \emph{affine geodesic}. An affine geodesic as defined here is similar
to the notion of a geodesic with an affine parameter.
\end{defn}

If $\overline{\boldsymbol{\gamma}}$ is an affine geodesic then obviously
\begin{equation}
\overline{\triangledown}{}_{1}\left(\dot{\overline{\triangledown}}{}_{1}\overline{\boldsymbol{\gamma}}\right)\left(x\right)=0\label{eq:vanDer}
\end{equation}
 for all $x\in\mathbb{R}$, and under natural assumptions (\ref{eq:vanDer})
conversely implies that $\overline{\boldsymbol{\gamma}}$ is an affine
geodesic.

\subsection{Decomposition of reduced derivatives of vector fields}

\begin{lem}
\label{lem:34}If $\Delta_{u}v$ exists then $\lim_{\tau\rightarrow0}p^{-1}\circ\left(p+\overline{\tau u}\right)\left(\frac{\boldsymbol{v}\left(p+\overline{\tau u}\right)-\boldsymbol{v}\left(p\right)}{\tau}\right)=\delta{}_{u}\boldsymbol{v}\left(p\right)$.
\end{lem}

\begin{proof}
Set $X\left(\tau\right)=p^{-1}\circ\left(p+\overline{\tau u}\right)\left(\frac{\boldsymbol{v}\left(p+\overline{\tau u}\right)-\boldsymbol{v}\left(p\right)}{\tau}\right)-\frac{\boldsymbol{v}\left(p+\overline{\tau u}\right)-\boldsymbol{v}\left(p\right)}{\tau}$
and $Y\left(\tau\right)=\frac{\boldsymbol{v}\left(p+\overline{\tau u}\right)-\boldsymbol{v}\left(p\right)}{\tau}$.
Using (\ref{eq:cont}), we obtain
\begin{align*}
\lim_{\tau\rightarrow0}p^{-1}\circ\left(p+\overline{\tau u}\right)\left(\frac{\boldsymbol{v}\left(p+\overline{\tau u}\right)-\boldsymbol{v}\left(p\right)}{\tau}\right) & =\lim_{\tau\rightarrow0}\left(X\left(t\right)+\left(Y\left(t\right)\right)\right)=\lim_{\tau\rightarrow0}X\left(t\right)+\lim_{\tau\rightarrow0}Y\left(t\right)\\
 & =\lim_{\tau\rightarrow0}Y\left(t\right)=\lim_{\tau\rightarrow0}\frac{\boldsymbol{v}\left(p+\overline{\tau u}\right)-\boldsymbol{v}\left(p\right)}{\tau}=\delta{}_{u}\boldsymbol{v}\left(p\right).\qedhere
\end{align*}
\end{proof}
\begin{prop}
\label{prop:p29}Consider vector spaces $U\subseteq V$, pointwise
affine systems $\mathscr{A}=\left(U,A,\boldsymbol{\mathfrak{a}}\right)$,
$\mathscr{B}=\left(V,V,\boldsymbol{\mathfrak{b}}\right)$ and a vector
field $\boldsymbol{v}:A\rightarrow V$ Then. \textup{
\begin{equation}
\nabla_{u}\boldsymbol{v}\left(p\right)=\delta{}_{u}\boldsymbol{v}\left(p\right)+\Delta{}_{u}\boldsymbol{v}\left(p\right).\label{eq:p30}
\end{equation}
 }
\end{prop}

\begin{proof}
Set $X\left(\tau\right)=\frac{p^{-1}\circ\left(p+\overline{\tau u}\right)\left(\boldsymbol{v}\left(p+\overline{\tau u}\right)-\boldsymbol{v}\left(p\right)\right)}{\tau}$
and $Y\left(\tau\right)=\frac{p^{-1}\circ\left(p+\overline{\tau u}\right)\left(\boldsymbol{v}\left(p\right)\right)-\boldsymbol{v}\left(p\right)}{\tau}$.
Using Lemma \ref{lem:34}, we obtain 
\begin{align*}
\nabla_{u}\boldsymbol{v}\left(p\right) & =\lim_{\tau\rightarrow0}\frac{p^{-1}\circ\left(p+\overline{\tau u}\right)\left(\boldsymbol{v}\left(p+\overline{\tau u}\right)\right)-\boldsymbol{v}\left(p\right)}{\tau}\\
 & =\lim_{\tau\rightarrow0}\left(X\left(\tau\right)+Y\left(\tau\right)\right)=\lim_{\tau\rightarrow0}X\left(\tau\right)+\lim_{\tau\rightarrow0}Y\left(\tau\right)\\
 & =\lim_{\tau\rightarrow0}p^{-1}\circ\left(p+\overline{\tau u}\right)\left(\frac{\boldsymbol{v}\left(p+\overline{\tau u}\right)-\boldsymbol{v}\left(p\right)}{\tau}\right)+\lim_{\tau\rightarrow0}\frac{p^{-1}\circ\left(p+\overline{\tau u}\right)\left(\boldsymbol{v}\left(p\right)\right)-\boldsymbol{v}\left(p\right)}{\tau}\\
 & =\delta{}_{u}\boldsymbol{v}\left(p\right)+\Delta{}_{u}\boldsymbol{v}\left(p\right).\qedhere
\end{align*}
\end{proof}
\begin{cor}
\label{cor:c34}Consider vector spaces $U\subseteq V$, pointwise
affine systems $\mathscr{A}=\left(U,U,\boldsymbol{\mathfrak{a}}\right)$,
$\mathscr{B}=\left(V,V,\boldsymbol{\mathfrak{b}}\right)$ and a vector
field $\boldsymbol{v}:U\rightarrow V$ Then. 
\begin{equation}
\triangledown{}_{u}\boldsymbol{v}\left(x\right)=\partial{}_{u}\boldsymbol{v}\left(x\right)+\Delta{}_{u}\boldsymbol{v}\left(x\right).\label{eq:c31-1}
\end{equation}
\end{cor}

\begin{proof}
Set $X\left(\tau\right)=\frac{p^{-1}\circ\left(p+\overline{\tau u}\right)\left(\boldsymbol{v}\left(p+\tau u\right)-\boldsymbol{v}\left(p\right)\right)}{\tau}$
and $Y\left(\tau\right)=\frac{p^{-1}\circ\left(p+\overline{\tau u}\right)\left(\boldsymbol{v}\left(p\right)\right)-\boldsymbol{v}\left(p\right)}{\tau}$.
\end{proof}

\subsection{\label{subsec:94}Decomposition of derivatives of covector fields}

Recall that $\delta{}_{u}\left(\varphi\left(\boldsymbol{v}\right)\right)\left(p\right)=\varphi\left(\nabla_{u}\boldsymbol{v}\right)\left(p\right)+\left(\delta_{u}\varphi\right)\left(\boldsymbol{v}\right)\left(p\right)$
by Corollary \ref{cor:40}, so equivalently 
\begin{equation}
\left(\delta_{u}\varphi\right)\left(\boldsymbol{v}\right)\left(p\right)=\delta{}_{u}\left(\varphi\left(\boldsymbol{v}\right)\right)\left(p\right)-\varphi\left(\nabla_{u}\boldsymbol{v}\right)\left(p\right).\label{eq:cuvf}
\end{equation}
This result is often used as a definition of the derivative $\delta_{u}\varphi$
of a covector field, requiring (\ref{eq:cuvf}) to hold for every
vector field $\boldsymbol{v}$ and point $p$.

We can also use (\ref{eq:cuvf}) to derive a decomposition of derivatives
of covector fields analogous to the decomposition of derivatives of
vector fields above. 
\begin{prop}
Consider a covector field $\varphi:A\rightarrow V^{*}$ and a vector
field $\boldsymbol{v}:A\rightarrow V$. If $\boldsymbol{v}\left(x\right)=v$
for all $x\in A$ then 
\begin{equation}
\left(\delta_{u}\varphi\right)\left(\boldsymbol{v}\right)\left(p\right)=\delta{}_{u}\left(\varphi\left(\boldsymbol{v}\right)\right)\left(p\right)-\varphi\left(\Delta{}_{u}\boldsymbol{v}\right)\left(p\right).\label{eq:p41}
\end{equation}
\end{prop}

\begin{proof}
By (\ref{eq:cuvf}) and Proposition \ref{prop:p29}, we have
\begin{align*}
\left(\delta_{u}\varphi\right)\left(\boldsymbol{v}\right)\left(p\right) & =\delta_{u}\left(\varphi\left(\boldsymbol{v}\right)\right)\left(p\right)-\varphi\left(\nabla_{u}\boldsymbol{v}\right)\left(p\right)\\
 & =\delta{}_{u}\left(\varphi\left(\boldsymbol{v}\right)\right)\left(p\right)-\varphi\left(\delta{}_{u}\boldsymbol{v}+\Delta{}_{u}\boldsymbol{v}\right)\left(p\right)\\
 & =\delta{}_{u}\left(\varphi\left(\boldsymbol{v}\right)\right)\left(p\right)-\varphi\left(\delta{}_{u}\boldsymbol{v}\right)\left(p\right)-\varphi\left(\Delta{}_{u}\boldsymbol{v}\right)\left(p\right),
\end{align*}
where $\boldsymbol{\boldsymbol{v}}\left(x\right)=v$ for all $x\in A$
implies $\delta{}_{u}\boldsymbol{v}=0$ (whereas this does not imply
that $\Delta{}_{u}\boldsymbol{v}=0$).
\end{proof}
\begin{cor}
If $\mathscr{A}=\left(U,U,\boldsymbol{\mathfrak{ci}}_{U}\right)$
and $\boldsymbol{v}\left(x\right)=v$ for all $x\in X$ then 
\begin{equation}
\left(\partial{}_{u}\varphi\right)\left(\boldsymbol{v}\right)\left(x\right)=\partial_{u}\left(\varphi\left(\boldsymbol{v}\right)\right)\left(x\right)-\varphi\left(\Delta{}_{u}v\right)\left(x\right).\label{eq:c42}
\end{equation}
\end{cor}

\begin{proof}
Use Corollary \ref{cor:40-1}.
\end{proof}

\subsection{The Lie derivative $\mathcal{L}_{\boldsymbol{u}}\boldsymbol{v}$
and torsion revisited}

By Definition \ref{def:53} , $T_{uv}(x)=\Delta_{u}v\left(x\right)-\Delta_{v}u\left(x\right)$,
and replacing the vectors $u,v$ by the vector fields $\boldsymbol{u},\boldsymbol{v}$
gives $T_{\boldsymbol{uv}}(x)=\Delta_{\boldsymbol{u}}\boldsymbol{v}\left(x\right)-\Delta_{\boldsymbol{v}}\boldsymbol{u}\left(x\right)$. 

As an \emph{Ansatz}, we define the \emph{Lie derivative} $\mathcal{L}_{\boldsymbol{u}}\boldsymbol{v}\left(x\right)$
of $\boldsymbol{v}$ along $\boldsymbol{u}$ at $x$ by setting
\begin{equation}
\mathcal{L}_{\boldsymbol{u}}\boldsymbol{v}\left(x\right):=\left[\left(\partial{}_{\boldsymbol{u}}\boldsymbol{v}\right)\left(x\right)\right]=\left(\partial{}_{\boldsymbol{u}}\boldsymbol{v}\right)\left(x\right)-\left(\partial{}_{\boldsymbol{v}}\,\boldsymbol{u}\right)\left(x\right).\label{eq:6542}
\end{equation}
 In view of (\ref{eq:6542}), $\mathcal{L}_{\boldsymbol{u}}\boldsymbol{v}\left(x\right)$
depends on $\boldsymbol{u}$ and $\boldsymbol{v}$, not only on $\boldsymbol{u}\left(x\right)$
and $\boldsymbol{v}\left(x\right)$.

By Proposition (\ref{prop:p29}), we have $\partial{}_{\boldsymbol{u}}\boldsymbol{v}\left(x\right)=\triangledown_{\boldsymbol{u}}\boldsymbol{v}\left(x\right)-\Delta{}_{\boldsymbol{u}}\boldsymbol{v}\left(x\right)$,
so it follows from (\ref{eq:6542}) that
\[
\mathcal{L}_{\boldsymbol{u}}\boldsymbol{v}=\left(\triangledown{}_{\boldsymbol{u}}\,\boldsymbol{v}-\Delta{}_{\boldsymbol{u}}\,\boldsymbol{v}\right)-\left(\triangledown{}_{\boldsymbol{v}}\,\boldsymbol{u}-\Delta{}_{\boldsymbol{v}}\,\boldsymbol{u}\right)=\triangledown_{\boldsymbol{u}}\boldsymbol{v}-\triangledown_{\boldsymbol{v}}\boldsymbol{u}-T_{\boldsymbol{uv}}.
\]
On the other hand, this result can be used to define torsion in terms
of derivatives, setting
\[
T_{\boldsymbol{uv}}=\triangledown{}_{\boldsymbol{u}}\boldsymbol{v}-\triangledown{}_{\boldsymbol{v}}\boldsymbol{u}-\mathcal{L}_{\boldsymbol{u}}\boldsymbol{v}.
\]
We have thus recovered the usual definition of torsion by defining
$\mathcal{L}_{\boldsymbol{u}}\boldsymbol{v}$ as in (\ref{eq:6542}).
However, this definition is somewhat unnatural, since it gives the
impression that $T_{\boldsymbol{uv}}\left(x\right)$ depends on $\boldsymbol{u}$,
$\boldsymbol{v}$ and $x$, whereas Definition \ref{def:53} makes
clear that $T_{\boldsymbol{uv}}\left(x\right)$ depends only on $\boldsymbol{u}\left(x\right)$,
$\boldsymbol{v}\left(x\right)$ and $x$.

\subsection{\label{subsec:s96}The Koszul connection $\nabla_{\boldsymbol{u}}\boldsymbol{v}$
and the pseudo-derivative $\Delta_{\boldsymbol{u}}\boldsymbol{v}$}

Consider a tuple of vector fields $\left(\boldsymbol{e}_{1},\ldots,\boldsymbol{e_{n}}\right)$
such that $\boldsymbol{e}_{i}\left(x\right)=e_{i}\in V$ for all $x\in X$,
where $\left(e_{1},\ldots,e_{n}\right)$ is a basis for $V$. Then
Proposition \ref{prop:p29} implies that
\[
\nabla_{\boldsymbol{e}_{i}}\boldsymbol{e}_{j}\left(x\right)=\delta{}_{\boldsymbol{e}_{i}}\boldsymbol{e}_{j}\left(x\right)+\Delta{}_{\boldsymbol{e}_{i}}\boldsymbol{e}_{j}\left(x\right)=\Delta{}_{\boldsymbol{e}_{i}}\boldsymbol{e}_{j}\left(x\right)
\]
for every $i,j=1,\ldots,n$ since each $\boldsymbol{e}_{j}$ is a
constant vector field. This equality shows that the derivative $\nabla_{\boldsymbol{u}}\boldsymbol{v}$
carries information about curvature, similar to $\Delta{}_{\boldsymbol{u}}\boldsymbol{v}$
(and $\Gamma_{ij}^{k}$), motivating the term ''connection''.

By Proposition \ref{prop:87}, Corollary \ref{prop:33} and other
results in this article, (1) -- (2) below hold when interpreting
$\boldsymbol{\nabla}_{\boldsymbol{u}}\boldsymbol{v}$ as $\nabla{}_{\boldsymbol{u}}\boldsymbol{v}$,
$\triangledown{}_{\boldsymbol{u}}\boldsymbol{v}$, $\delta_{\boldsymbol{u}}\boldsymbol{v}$
or $\partial_{\boldsymbol{u}}\boldsymbol{v}$, and (3) holds at least
when interpreted as $\nabla\delta{}_{u}\left(\phi\boldsymbol{v}\right)\left(p\right)=\phi\left(\nabla_{u}\boldsymbol{v}\right)\left(p\right)+\left(\delta{}_{u}\phi\right)\boldsymbol{v}\left(p\right)$
or $\triangledown\partial_{u}\left(\phi\boldsymbol{v}\right)\left(p\right)=\phi\left(\triangledown{}_{u}\boldsymbol{v}\right)\left(p\right)+\left(\partial{}_{u}\phi\right)\boldsymbol{v}\left(p\right)$.
\begin{enumerate}
\item %
$\boldsymbol{\nabla}_{\left(\phi\boldsymbol{u}+\boldsymbol{u}'\right)}\boldsymbol{v}=\phi\boldsymbol{\nabla}_{\boldsymbol{u}}\boldsymbol{v}+\boldsymbol{\nabla}_{\boldsymbol{u}'}\boldsymbol{v}$
for any scalar function $\phi$ (see Section \ref{subsec:75}).
\item %
$\boldsymbol{\nabla}_{\boldsymbol{u}}\left(\lambda\boldsymbol{v}+\boldsymbol{v}'\right)=\lambda\boldsymbol{\nabla}_{\boldsymbol{u}}\boldsymbol{v}+\boldsymbol{\nabla}_{\boldsymbol{u}}\boldsymbol{v}'$
for any $\lambda\in K$ (see Section \ref{subsec:75}).
\item %
$\boldsymbol{\nabla}_{\boldsymbol{u}}\left(\phi\boldsymbol{v}\right)=\phi\left(\boldsymbol{\nabla}_{\boldsymbol{u}}\boldsymbol{v}\right)+\left(\boldsymbol{\nabla}_{\boldsymbol{u}}\phi\right)\boldsymbol{v}$
for any scalar function $\phi$. 
\end{enumerate}
Conversely, it has been shown that there is only one function $\nabla:\left(\boldsymbol{u},\boldsymbol{v}\right)\mapsto\nabla_{\boldsymbol{u}}\boldsymbol{v}$
satisfying these conditions, so a connection can also be defined as
a function $\nabla:\left(\boldsymbol{u},\boldsymbol{v}\right)\mapsto\nabla_{\boldsymbol{u}}\boldsymbol{v}$
satisfying (1) -- (3). This is the so-called Koszul connection.

Note, though, that we cannot equate the Koszul connection $\nabla_{\boldsymbol{u}}\boldsymbol{v}$
with the pseudo-derivative $\Delta{}_{\boldsymbol{u}}\boldsymbol{v}$
or the corresponding affine connection $\boldsymbol{C}$ (Definition
\ref{def:45}) since we do not have $\nabla_{\boldsymbol{u}}\boldsymbol{v}=\Delta{}_{\boldsymbol{u}}\boldsymbol{v}$
unless $\boldsymbol{v}$ is a constant vector field so that $\delta{}_{\boldsymbol{u}}\boldsymbol{v}=0$.
This is related to the fact that $\Delta_{\boldsymbol{u}}\boldsymbol{v}\left(x\right)=\Delta_{\boldsymbol{u}\left(x\right)}\boldsymbol{v}\left(x\right)\left(x\right)$
for all $x$, meaning that $\Delta{}_{\boldsymbol{u}}\boldsymbol{v}$
does not depend on $\boldsymbol{u}\left(x'\right)$ or $\boldsymbol{v}\left(x'\right)$
for any $x'\neq x$. In fact, we do not need the notion of a vector
field to define a connection and specify the curvature of a pointwise
affine space since we can use $\Delta_{u}v$ instead of $\Delta{}_{\boldsymbol{u}}\boldsymbol{v}$,
$\nabla_{\boldsymbol{\boldsymbol{u}}}\boldsymbol{v}$ etc.

\section{Final remarks}

For practical purposes, the crucial innovation in the approach to
differential geometry suggested in this article is the pseudo-derivative
$\Delta_{u}v$. As suggested by the similarity of relation (\ref{eq:curvi}),
\begin{gather*}
\nabla_{u}\boldsymbol{v}\left(x\right)\sim_{\overline{C}}\:\overline{u}^{i}\frac{\partial\overline{\boldsymbol{v}}^{k}}{\partial\overline{x}^{i}}\left(x\right)+\overline{u}^{i}\overline{\boldsymbol{v}}^{j}\Gamma_{ij}^{k}\left(x\right),
\end{gather*}
and equation (\ref{eq:c31-1}), 
\begin{gather*}
\triangledown{}_{u}\boldsymbol{v}\left(x\right)=\partial{}_{u}\boldsymbol{v}\left(x\right)+\Delta{}_{u}\boldsymbol{v}\left(x\right),
\end{gather*}
$\Delta_{u}v$ is a coordinate-free counterpart to $u^{i}v^{j}\Gamma_{ij}^{k}$.
Simultaneously, $\Delta_{u}v$ is closely related to the Koszul connection
$\nabla{}_{u}\boldsymbol{v}$ as discussed in Section \ref{subsec:s96}.
Note that $\Delta_{u}v$ allows simple, natural definitions of several
central notions because $\Delta_{u}v$ measures the curvature of a
space directly, not indirectly via its effect on vector fields on
that space. For example, torsion is defined by $T_{uv}=\Delta_{u}v-\Delta_{v}u$
rather than the usual $T_{\boldsymbol{uv}}=\nabla{}_{\boldsymbol{u}}\boldsymbol{v}-\nabla{}_{\boldsymbol{v}}\boldsymbol{u}-\mathcal{L}_{\boldsymbol{u}}\boldsymbol{v}$.

From a more abstract point of view, pointwise affine spaces generalize
affine spaces, while manifolds generalize tuple spaces such as $\mathbb{R}^{n}$.
Accordingly, the two approaches employ different strategies to specify
curvature and derivatives.

Specifically, for each point $p$ of a manifold $M$ there is a vector
space $T_{p}M$, the tangent space at $p$. All tangent spaces for
a particular manifold are isomorphic, but not canonically isomorphic;
instead we equip $M$ with a choice of isomorphisms $\mathscr{I}{}_{\!pp'}:T_{p}M\rightarrow T_{p'}M$
for certain pairs of points $p,p'$. While we cannot subtract vectors
in different vector spaces, the difference 
\[
\boldsymbol{v}\left(p\right)-\mathscr{I}{}_{\!pp'}^{-1}\left(\boldsymbol{v}\left(p'\right)\right)
\]
exists for any vector field $\boldsymbol{v}$ provided that $\mathscr{I}{}_{\!pp'}$
exists, and we can use expressions of this form to define a derivative
$\nabla_{u}\boldsymbol{v}$ of $\boldsymbol{v}$ when a choice of
$\mathscr{I}{}_{\!pp'}$ has been made for each point pair $p,p'$
such that $p'$ is sufficiently close to $p$. Such derivatives can
be used to specify the curvature of a manifold. 

For pointwise affine spaces, there is a single group $V$ that acts
on the point space $A$. For each point, there is now a particular
way in which $V$ acts on $A$ rather than a particular tangent space.
In this case, $\boldsymbol{v}\left(p\right)$ and $\boldsymbol{v}\left(p'\right)$
belong to the same vector space, isomorphic, in possibly different
ways for different points of $A$, to the vector space of translations
on $A$. Thus, the difference $\boldsymbol{v}\left(p\right)-\boldsymbol{v}\left(p'\right)$
always exists and can be used to define derivatives, so constructions
such as tangent spaces, tangent bundles and sections of tangent bundles
are redundant in this approach. For example, a vector field is defined
simply as a function on a point space with values in a fixed vector
space. Moreover, an automorphism on $V$ of the form
\[
v\mapsto\boldsymbol{\mathfrak{a}}\left(p\right)^{-1}\circ\boldsymbol{\mathfrak{a}}\left(p'\right)\left(v\right)
\]
can be used to define $\Delta_{u}v\left(p\right)$ and specify the
affine curvature at $p$ of a pointwise affine space.

The simplicity of pointwise affine spaces is partly a consequence
of reduced scope and generality rather than of simplification of unnecessarily
cumbersome definitions. We obviously obtain a more complicated space
by adding a metric to a pointwise affine space, and we nave not considered
general tensors. Moreover, a pointwise affine space is a special,
uncomplicated manifold. A general manifold has a global ''shape''
which cannot be changed by changing its curvature locally: it may
be a surface, a sphere, a torus etc. A pointwise affine space is a
manifold with a trivial shape. To obtain a manifold with a non-trivial
shape such as a sphere we can ''glue together'' more than one trivial-shape
manifolds. 

Generalizing the notion of a pointwise affine space with the definition
of a manifold or a Cartan geometry as a model, we first note that
an open subset of a topological space can be equipped with the structure
of a pointwise affine space. Given a vector space $V$ and a point
space $A$ with a suitable topology, we can thus assign to every $p\in A$
a pointwise affine system $\left(V,U_{p},\boldsymbol{\mathfrak{a}}_{p}\right)$,
where $U_{p}$ is an open neighborhood of $p$ and $\boldsymbol{\mathfrak{a}}_{p}$
is an affine action field on $U_{p}$. Thus, for every $p\in A$ it
is possible to obtain a pseudo-derivative or derivative at $p$ by
using $\boldsymbol{\mathfrak{a}}_{p}$. It is not immediately clear,
though, what compatibility conditions, if any, that would have to
be imposed on the set $\left\{ \boldsymbol{\mathfrak{a}}_{p}\mid p\in A\right\} $
to make a definition along these lines work. This problem will be
left open here.%

Anyhow, while a pointwise affine space corresponds only to a special
manifold which can be covered by a single coordinate chart, initially
regarding curved spaces as pointwise affine spaces may make them easier
to understand. Also, in many applications the general theory is not
needed.


\begin{thebibliography}{1}
\bibitem{key-1}Weyl, Hermann (1918). \emph{Raum - Zeit - Materie}.
Berlin: Springer.

\bibitem{key-2}Grassmann, Hermann (1844). \emph{Die Lineale Ausdehnungslehre}.
Leipzig: Otto Wigand.

\bibitem{key-3}Peano, G. (1888). \emph{Calcolo geometrico secundo
l'Ausdehnungslehre di H. Grassmann e precedutto alle operazioni della
logica deduttiva}, Turin: Fratelli Bocca Editori, 1888.

\bibitem{key-4}Dorier, J-L. (1995). ''A General Outline of the Genesis
of Vector Space Theory''. \emph{Historia Mathematica} \textbf{22},
227-261.

\bibitem{key-5}Ricci, Gregorio; Levi-Civita (1900), Tullio. ''Méthodes
de calcul différentiel absolu et leurs applications''. \emph{Mathematische
Annalen}, \textbf{54} (1--2), 125--201.

\bibitem{key-6}Einstein, A.; Grossmann, M. (1913). \textquotedbl Entwurf
einer verallgemeinerten Relativitätstheorie und einer Theorie der
Gravitation\textquotedbl . \emph{Zeitschrift für Mathematik und Physik}.\textbf{
62}, 225--261.

\bibitem{key-7}Jonsson, Dan (2016). ''Notes on Actions of Sets and
Groups and Generalized Affine Spaces''\emph{.} arXiv:1607.03890 {[}math.GR{]}.

\bibitem{key-8}Berger, Marcel (1978). \emph{Geometry I.} Berlin:
Springer.

\end{thebibliography}
\end{document}